\documentclass{article}
\usepackage[lang=british]{ems-rmi}
\usepackage{slashed}
\usepackage{mathrsfs, esint, dsfont}
\usepackage{enumitem}
\usepackage{amssymb}
\usepackage{mathrsfs}
\usepackage{setspace}
\usepackage{subcaption}
\usepackage{bbm}

\usepackage{tocloft}

\numberwithin{equation}{section}
\newtheorem{theorem}{Theorem}[section]
\newtheorem{proposition}[theorem]{Proposition}
\newtheorem{lemma}[theorem]{Lemma}
\newtheorem{corollary}[theorem]{Corollary}

\theoremstyle{definition}

\newtheorem{remark}[theorem]{Remark}

\newcommand\1{{\mathds 1}}

\def\bbI{{\mathbb I}}

\def\N{{\mathbb N}}

\def\R{{\mathbb R}}
\def\C{{\mathbb C}}

\def\SS{{\mathbb S}}

\def\Z{{\mathbb Z}}

\def\rd{{\mathrm{d}}}
\def\re{{\mathrm{e}}}
\def\ri{{\mathrm{i}}}

\def\cB{{\mathcal B}}

\def\cE{{\mathcal E}}
\def\cF{{\mathcal F}}

\def\cM{{\mathcal M}}
\def\cN{{\mathcal N}}

\def\cS{{\mathcal S}}

\def\fS{{\mathfrak S}}

\newcommand{\sC}{\mathscr{C}}

\newcommand{\Id}{\mathbbm 1}
\newcommand{\msc}[1]{\href{https://mathscinet.ams.org/mathscinet/search/mscdoc.html?code=#1}{#1}}
\newcommand{\be}[1]{\begin{equation}\label{#1}}
\newcommand{\ee}{\end{equation}}
\renewcommand{\(}{\left(}
\renewcommand{\)}{\right)}
\newcommand{\ird}[1]{\int_{\R^d}{#1}\,\rd x}
\newcommand{\ir}[1]{\int_{\R}{#1}\,\rd x}
\newcommand{\nrm}[2]{\|{#1}\|_{#2}}
\newcommand{\Dom}{\operatorname{Dom}}
\newcommand{\Dirac}{\slashed{\mathcal{D}}_m}
\def\ri{{\mathrm{i}}}
\def\eps{\varepsilon}
\def\rad{{\rm rad}}
\newcommand{\aalpha}{\mathsf a}
\newcommand{\bbeta}{\mathsf b}
\newcommand{\st}{\mathsf t}
\newcommand{\sV}{\mathsf W}

\begin{document}
\title{Keller and Lieb-Thirring estimates of the eigenvalues\\ in the gap of Dirac operators}
\titlemark{Keller-Lieb-Thirring estimates and Dirac operators}

\emsauthor1{Jean Dolbeault}{J.~Dolbeault}
\emsauthor2{David Gontier}{D.~Gontier}
\emsauthor{3}{Fabio Pizzichillo}{F.~Pizzichillo}
\emsauthor{4}{Hanne Van Den Bosch}{H.~Van Den Bosch}

\emsaffil1{CEREMADE (CNRS UMR n$^\circ$ 7534), PSL University, Universit\'e Paris-Dauphine, Place de Lattre de Tassigny, 75775 Paris 16, France \email{dolbeaul@ceremade.dauphine.fr}}
\emsaffil2{CEREMADE (CNRS UMR n$^\circ$ 7534), PSL University, Universit\'e Paris-Dauphine, Place de Lattre de Tassigny, 75775 Paris 16, France \& ENS/PSL University, DMA, F-75005, Paris, France \email{gontier@ceremade.dauphine.fr}}
\emsaffil{3}{
Departamento de Matem\'atica e Inform\'atica aplicadas a las Ingenierías Civil y Naval, Universidad Polit\'ecnica de Madrid,
Escuela T\'ecnica Superior de Ingenieros de Caminos, Canales y Puertos,
Calle Profesor Aranguren, 3.
28040, Madrid, Spain\email{fabio.pizzichillo@upm.es}}
\emsaffil{4}{Departamento de Ingenier\'{\i}a Matem\'atica \& Center for Mathematical Modeling, Facultad de Ciencias F\'{\i}sicas y Matem\'aticas, Universidad de Chile and CNRS IRL 2807, Beauchef 851, Piso 5, Santiago, Chile \email{hvdbosch@dim.uchile.cl}}

\classification[\msc{49R05}, \msc{49J35}, \msc{47A75}, \msc{47B25}]{\msc{81Q10}}


\begin{abstract}
We estimate the lowest eigenvalue in the gap of the essential spectrum of a Dirac operator with mass in terms of a Lebesgue norm of the potential. Such a bound is the counterpart for Dirac operators of the \emph{Keller estimates} for the Schr\"odinger operator, which are equivalent to Gagliardo-Nirenberg-Sobolev interpolation inequalities. Domain, self-adjointness, optimality and critical values of the norms are addressed, while the optimal potential is given by a Dirac equation with a Kerr nonlinearity. A new \emph{critical} bound appears, which is the smallest value of the norm of the potential for which eigenvalues may reach the bottom of the gap in the essential spectrum. The Keller estimate is then extended to a Lieb-Thirring inequality for the eigenvalues in the gap. Most of our result are established in the \emph{Birman-Schwinger} reformulation.
\end{abstract}

\maketitle

\vspace*{-40pt}
\begin{spacing}{0.75}
{\renewcommand{\contentsname}{}\small\tableofcontents}
\addtocontents{toc}{\protect\setcounter{tocdepth}{2}}
\end{spacing}
\thispagestyle{empty}
\pagebreak

\section{Introduction and main results}

In 1961, J.B.~Keller established in~\cite{MR121101} the expression of the potential which minimizes the lowest eigenvalue, or \emph{ground state}, $\lambda_S(V)$ of the Schr\"odinger operator $-\Delta-V$ in dimension $d=1$, under a constraint on the Lebesgue norm
\[
\nrm Vp=\(\ird{|V|^p}\)^{1/p}
\]
of exponent $p$ of $V$. This estimate was later extended in~\cite{Lieb-Thirring76} by E.H.~Lieb and W.~Thirring to higher dimensions and to a sum of the lowest eigenvalues. During the last forty years, various refinements were published. As an example, we quote stability results for $\lambda_S(V)$ proved in~\cite{MR3177378} by E.A.~Carlen, R.L.~Frank, and E.H.~Lieb. Although Dirac operators inherit many qualitative properties of Schr\"odinger operators, dealing with Dirac operators turns out to be a delicate issue. 

If $\Dirac$ denotes the free Dirac operator and $V$ is a non-negative valued function, \hbox{$\Dirac - V$} is not bounded from below. One is actually interested in the lowest eigenvalue $\lambda_D(V)$ in the essential gap $(-m\,c^2,m\,c^2)$, where $m$ denotes the mass and $c$ the speed of light. We shall speak of $\lambda_D(V)$ as the \emph{ground state} energy of $\Dirac-V$. In the standard setting, it is expected that $\lambda_D(V)-m\,c^2$ converges to $\lambda_S(V)$ in the non-relativistic limit, \emph{i.e.}, as $c\to+\infty$. It is therefore a natural question to estimate $\lambda_D(V)$ in terms of $\nrm Vp$ and identify the corresponding optimal potential. This question is the main purpose of our paper. A new \emph{critical value} appears, which corresponds to the smallest value of $\nrm Vp$ for which $\lambda_D(V)$ reaches, for some potential $V \ge 0$, the lower end of the essential gap $-\,m\,c^2$. In a linear setting, a similar question has been raised in~\cite{esteban2021diraccoulombII,ELS2021}, where the authors find a critical value $\nu_1$ so that $\lambda_D\big(\mu * | \cdot |^{-1}\big) > -\,m\,c^2$ for all positive measures $\mu$ with $\mu(\R^3) < \nu_1$, with $2/\big(\pi/2+2/\pi\big)<\nu_1\le1$. Going back to~\cite{MR1761368,DDEV,MR2091354}, it is known that \emph{Hardy inequalities} play an essential role in the analysis of the spectrum of Dirac-Coulomb operators. In the present article, except for the case $p=d=1$, we rather find a nonlinear functional inequality of Gagliardo-Nirenberg-Sobolev nature, instead of a Hardy inequality (see comments in Appendix~\ref{appendix:GNS-2}).

It is possible to characterize the eigenvalues of $\Dirac-V$ in the gap by a \emph{min-max principle} according to~\cite{MR2239275,MR1761368,DES2021} but this raises delicate issues involving the domain of the operator and its self-adjoint extensions addressed respectively in~\cite{SchSolTok20,DES2021,MR2370231,MR2466677,MR3960263}. Applied with a Coulombian potential $V$, the method gives rise, after the maximising step in the min-max method, to a lower bounded quadratic form which amounts to a kind of Hardy inequality for the upper component: see~\cite{MR2091354,DDEV,MR2379440} for details. The same strategy applies to a general potential~$V$ under a constraint on $\nrm Vp$, except that the Keller type bound on $\lambda_D(V)$ is given by an implicit condition: see Appendix~\ref{appendix:GNS}. The optimal potential solves a nonlinear Dirac equation with Kerr-type nonlinearity. For the two-dimensional case, this equation has been studied in~\cite{MR4068289,MR3705703,MR3859463,MR3794033} by W.~Borrelli. In the one-dimensional case, the solution is explicit, which allows us to identify it as in the case of the Schr\"odinger operator studied in~\cite{MR121101}. Alternatively to the min-max principle, the properties of the \emph{Birman-Schwinger operator} corresponding to $\Dirac-V$ allows us to characterize $\lambda_D(V)$ and, except in Appendix~\ref{appendix:GNS}, we will adopt this point of view.

\medskip The \emph{Keller-Lieb-Thirring inequality for a Schr\"odinger operator} goes as follows. Let us assume that $q>2$, with $q<2^*:=2\,d/(d-2)$ if $d\ge3$, and let $\vartheta=d\,(q-2)/(2\,q)$. For any function $u\in\mathrm H^1(\R^d)$, the Gagliardo-Nirenberg-Sobolev inequality
\[
\nrm{\nabla u}2^\vartheta\,\nrm u2^{1-\vartheta}\ge\mathscr C_q\,\nrm uq
\]
can be rewritten in the non-scale invariant form as
\be{ScaledGN}
\forall\,(\lambda,u)\in(0,+\infty)\times\mathrm H^1(\R^d)\,,\quad\nrm{\nabla u}2^2+\lambda\,\nrm u2^2\ge\mathsf C_q\,\lambda^{1-\vartheta}\,\nrm uq^2
\ee
with an optimal constant $\mathsf C_q$ such that $\mathscr C_q^2=\vartheta^\vartheta(1-\vartheta)^{1-\vartheta}\,\mathsf C_q$. The equivalence of the two forms can be recovered by optimizing on $\lambda$ in~\eqref{ScaledGN}. There is also an inequality which is dual of~\eqref{ScaledGN} and goes as follows. Consider a potential $V\in\mathrm L^p(\R^d)$. Using H\"older's inequality with exponents $p$ and $q$ such that $1/p+2/q=1$ and $p>d/2$, and taking $\lambda$ so that $\mathsf C_q\,\lambda^{1-\vartheta}=\nrm Vp$, we deduce from~\eqref{ScaledGN} that
\[
\ird{|\nabla u|^2}-\ird{V\,|u|^2}\ge\nrm{\nabla u}2^2-\nrm Vp\,\nrm u q^2\ge-\(\mathsf C_q^{-1}\,\nrm Vp\)^\frac1{1-\vartheta}\,\nrm u 2^2\,.
\]
This is the Keller-Lieb-Thirring estimate for $-\Delta-V$, \emph{i.e.},
\be{KLT}
\forall\,V\in\mathrm L^p(\R^d)\,,\quad0\le\lambda_S^-(V)\le\mathsf K_p\,\nrm Vp^\eta
\ee
where $\eta:=1/(1-\vartheta)=2\,p/(2\,p-d)$ and $\lambda^-:=\max(0,-\lambda)$ denotes the negative part of~$\lambda$. See~\cite{DolEsLa-APDE2014,Dolbeault06082014,Dolbeault2018} for details. An optimization on~$V$ shows that~\eqref{ScaledGN} and~\eqref{KLT} are equivalent. The optimal constant in~\eqref{KLT} is $\mathsf K_p=\mathsf C_q^{-\eta}$. In addition, for all $\lambda >0$, if $u$ is a radial positive solution of
\be{eq:def:NLS_Intro}
-\Delta u-u^{\frac{p+1}{p-1}}=-\,\lambda\,u\,,
\ee
then $(u, \lambda)$ is an optimal pair for~\eqref{ScaledGN}, and $V :=u^{q-2}=u^{2/(p-1)}$ is an optimal potential for~\eqref{KLT}, which moreover satisfies $\lambda_S(V)=-\,\lambda$.
It turns out that the solution of~\eqref{eq:def:NLS_Intro} is unique up to translations according to~\cite{Coffman-72, Kwong-89, McLeod-93} and can be explicitly computed if $d=1$: see~\cite{MR121101}, or~\cite{Dolbeault06082014} and references therein for additional related results.

\medskip In order to state a \emph{Keller-Lieb-Thirring inequality for the Dirac operator}, we need some definitions and preliminary properties. Let us start with the free Dirac operator on $\R^d$. We refer to~\cite{MR1219537} for a comprehensive list of results and properties. For simplicity, we choose units in which $c=1$, except in Appendix~\ref{appendix:GNS} in which we consider the non-relativistic limit as $c\to+\infty$. Let $d\ge 1$ and set $N :=2^{\lfloor(d+1)/2\rfloor}$ where $\lfloor x\rfloor=\max\{n\in\Z\,:\,n\le x\}$ denotes the integer part of $x$. Let $\alpha_1, \cdots, \alpha_d$ and~$\beta$ be $N \times N$ Hermitian matrices satisfying the following anti-commutation rules
\be{eq:def:alphai_beta}
\forall\,j,\,k=1,\dots,d\,,\quad\begin{cases}
\alpha_j\,\alpha_k+\alpha_k\,\alpha_j=2\,\delta_{jk}\,\bbI_N\\
\alpha_j\,\beta+\beta\,\alpha_j=0\\
\beta^2=\bbI_N
\end{cases}
\ee
where $\delta_{jk}$ denotes the Kronecker symbol and $\bbI_N$ is the $N\times N$ identity matrix. See, \emph{e.g.},~\cite{Friedrich_2000} for an existence result for such matrices. The \emph{free Dirac operator} in dimension~$d$ is defined by
\[
\Dirac:=\sum_{j=1}^d\alpha_j\,(-\,\ri\,\partial_j)+m\,\beta=\boldsymbol\alpha\cdot(-\,\ri\,\nabla)+m\,\beta
\]
where we consider Cartesian coordinates $(x_1,\dots,x_d)$, $\partial_j:=\partial/\partial x_j$ and $\boldsymbol\alpha=(\alpha_k)_{k=1,\dots,d}$. With the Pauli matrices
\[\textstyle
\sigma_1:=\begin{pmatrix}0&1\\1&0\end{pmatrix}\,,\quad\sigma_2:=\begin{pmatrix}0&-\,\ri\\\ri&0\end{pmatrix}\quad\mbox{and}\quad\sigma_3:=\begin{pmatrix}1&0\\0&-1\end{pmatrix}\,,
\]
explicit expressions of $\Dirac$ are given
\begin{enumerate}[label=\emph{(\roman*)}]
\item in dimension $d=1$, by $\boldsymbol\alpha=\sigma_2$ and $\beta=\sigma_3$ so that
\[
\Dirac:=\sigma_2\,(-\,\ri\,\partial_1)+m\,\sigma_3\,,
\]
\item in dimension $d=2$, by $\boldsymbol\alpha=(\sigma_j)_{j=1,2}$ and $\beta=\sigma_3$ so that
\[
\Dirac:=\sum_{j=1}^2\sigma_j (-\,\ri\,\partial_j)+m\,\sigma_3\,,
\]
\item in dimension $d=3$, by $\boldsymbol\alpha=(\alpha_k)_{k=1,2,3}$ and $\beta$ such that
\[
\alpha_k:=\begin{pmatrix}0&\sigma_k\\\sigma_k&0\end{pmatrix}\quad\mbox{and}\quad\beta:=\begin{pmatrix}\bbI_2&0\\0&-\bbI_2\end{pmatrix}\,.
\]
\end{enumerate}
The free Dirac operator satisfies $\Dirac^2=-\,\Delta+m^2$. It is self-adjoint on $\mathrm L^2(\R^d,\C^N)$, with domain
\[
\Dom(\Dirac)=\mathrm H^1(\R^d,\C^N)
\]
and spectrum
\[
\sigma(\Dirac)=\sigma_{\rm ess}(\Dirac)=(-\infty,-m]\cup[m,+\infty)\,.
\]
Next we consider \emph{Dirac operators $\Dirac-V$ with potentials} $V\in\mathrm L^p(\R^d,\R^+)$ where the notation $\Dirac-V$ denotes $\Dirac-V\mathbb{I}_N$. When switching on a potential $V$, we expect that some eigenvalues of $\Dirac-V$ emerge from the upper essential spectrum $[m,+\infty)$. We shall prove in Section~\ref{Sec:Properties} that $\Dirac-V$ can be defined as a self-adjoint operator with \emph{essential spectrum} $\sigma_{\rm ess}(\Dirac-V)=\sigma_{\rm ess}(\Dirac)$. This allows us to define the \emph{ground state} $\lambda_D(V)$ as the lowest eigenvalue in the gap $(-\,m,m)$.

\medskip Our first result states that the \emph{ground state} is bounded by a function of $\nrm Vp$. Let
\be{eq:min_vap_Dirac}
\Lambda_D(\alpha,p):=\inf\Big\{\lambda_D(V)\,:\,V\in\mathrm L^p(\R^d,\R^+)\;\mbox{and}\;\nrm Vp=\alpha\Big\}\,.
\ee
\begin{theorem}\label{Thm:Main1} Assume that $p\ge d\ge 1$. There exists $\alpha_\star(p)>0$ such that the map $\alpha\mapsto\Lambda_D(\alpha,p)$ defined on $\big[0,\alpha_\star(p)\big)$ is continuous, strictly decreasing, takes values in $(-\,m,m]$, and such that
\[\label{Limits}
\lim_{\alpha\to0_+}\Lambda_D(\alpha,p)=m\quad\mbox{and}\quad\lim_{\alpha\to\alpha_\star(p)}\Lambda_D(\alpha,p)=-\,m\,.
\]
Moreover, if $(p,d)\neq(1,1)$, the infimum~\eqref{eq:min_vap_Dirac} is attained on $\big(0,\alpha_\star(p)\big)$ and
\[
\forall\,\alpha\in\big(0,\alpha_\star(p)\big)\,,\quad \Lambda_D(\alpha,p)=\lambda_D(V_{\alpha,p})
\]
where $V_{\alpha,p}=|\Psi|^{2/(p-1)}$, and $\Psi\in\mathrm L^2(\R^d,\C^N)$ solves the \emph{nonlinear Dirac equation}
\be{eq:def:NLD_Intro}
\Dirac\Psi-|\Psi|^\frac2{p-1}\,\Psi=\Lambda_D(\alpha,p)\,\Psi
\ee
and satisfies the constraint $\ird{|\Psi|^{2\,p/(p-1)}}=\nrm{V_{\alpha,p}}p^p=\alpha^p$.
\end{theorem}
\noindent The proof of Theorem~\ref{Thm:Main1} is given in Section~\ref{sec:proof:alpha} and relies on the properties of the inverse map of $\alpha \mapsto \Lambda_D(\alpha, p)$ defined by
\be{eq:def:alpha0}
\alpha_D(\lambda,p):=\inf\Big\{\nrm Vp\,:\,V\in\mathrm L^p(\R^d,\R^+)\;\mbox{and}\;\lambda_D(V)=\lambda\Big\}\,.
\ee
The critical value is $\alpha_\star(p)=\lim_{\lambda\to(-m)_+}\alpha_D(\lambda,p)$. It is such that
\[
\lim_{\alpha\to\alpha_\star(p)_-}\lambda_D(V_{\alpha,p})=-\,m
\]
and this limit is the upper bound of the lower essential spectrum $(-\infty,-m]$ or, equivalently, the lower end of the gap. For sake of simplicity, we adopt the convention that $\alpha_\star(p)=\alpha_D(-\,m,p)$. In the subcritical range of potentials, a simple consequence of Theorem~\ref{Thm:Main1} is the following \emph{Keller-Lieb-Thirring estimate} for the Dirac operator \hbox{$\Dirac-V$}.
\begin{corollary}\label{Cor:KLT} Assume $p\ge d\ge 1$. For all $V\in\mathrm L^p(\R^d,\R^+)$ with $\nrm Vp<\alpha_\star(p)$, we
have the optimal bound
\be{KLT-Dirac}
-\,m\le\Lambda_D\(\nrm Vp,p\)\le\lambda_D(V)\le m\,.
\ee
If $(p,d)\neq(1,1)$, then $V_{\alpha,p}$ as in Theorem~\ref{Thm:Main1} realizes the equality case, \emph{i.e.}, $\lambda_D(V_{\alpha,p})=\Lambda_D\(\alpha,p\)$.
\end{corollary}
\noindent Some plots of $\alpha \mapsto \Lambda_D(\alpha, p)$ are displayed in Fig.~\ref{fig:alphac_d1} (Right).

\medskip The nonlinear Dirac equation~\eqref{eq:def:NLD_Intro} plays for the Dirac operator \hbox{$\Dirac-V$} the same role as~\eqref{eq:def:NLS_Intro} for the Schr\"odinger operator $-\Delta+V$. However, $\Lambda_D(\alpha,V)$ is not obtained as the infimum but as a critical point of a Rayleigh quotient with infinitely many negative directions corresponding to a min-max principle (see~\cite{MR1761368}) and for this reason there is no simple interpolation inequality such as~\eqref{ScaledGN} in the case the Dirac operator. A more involved functional inequality holds: see Appendix~\ref{appendix:GNS}.

Nonlinear Dirac equations have been introduced to model extended fermions, as effective operators for nonlinear effects in graphene-like materials or Bose-Einstein condensates: see~\cite[Section~1.6]{Esteban_2008} and~\cite[Introduction]{MR3705703} for an introduction to the literature. Since the spinors in the Dirac equation have at least two components, many types of nonlinearities can be considered (see,~\emph{e.g.},~\cite{Pelinovsky-10} and references therein) and give rise to various phenomena. For instance, localized solutions to a nonlinear equation of the form 
\[
\Dirac \Psi-G(\Psi)=\lambda\,\Psi
\]
for some function $G : \C^N \mapsto \C^N$ correspond to solitary wave solutions to the time-depen\-dent nonlinear Dirac equation and have attracted considerable attention: see, \emph{e.g.},~\cite{CazenaveVasquez_1986, BaladaneCazenaveDouady_1988, Merle_1988, Esteban_1995}.

It is a common assumption to consider a nonlinearity that preserves Lorentz, or particle-hole, symmetry. Such a non-linearity takes the form
\be{eq:Soler_type}
\Dirac\Psi-F\big(\left\langle\Psi,\beta\,\Psi\right\rangle_{\C^N}\big)\,\Psi=\lambda\,\Psi
\ee
and is called the \emph{Soler-type nonlinearity}. The Soler nonlinearity formally appears when minimizing the first positive eigenvalue of $\Dirac-\beta\,V$ but will not be studied in this paper. In contrast, the nonlinearity that appears in~\eqref{eq:def:NLD_Intro} is of the form $F(\langle \Psi,\Psi\rangle_{\C^N})\,\Psi$, which is sometimes called a \emph{Kerr-type nonlinearity} as in~\cite{MR3705703}, apparently by extension of the cubic nonlinearity used in optics. Existence of localized solutions for~\eqref{eq:def:NLD_Intro} is studied in~\cite{MR3705703} in the critical exponent case $p=d=2$, and in~\cite{MR4068289} in the critical exponent case $p=d$ for all dimensions $d\in\N$ with \hbox{$m=0$}. Our results give an independent proof of the existence of a localized solution.

In~\cite{BaladaneCazenaveDouady_1988, Esteban_1995}, the authors proved that equations of the form~\eqref{eq:Soler_type} have many solutions if $d\ge 2$ by looking for solutions of~\eqref{eq:Soler_type} in subspaces of fixed angular momentum. It seems that similar techniques could also be applied to~\eqref{eq:def:NLD_Intro}. While it is reasonable to expect that the optimal potential is radially symmetric and the corresponding \emph{ground state} $\Psi$ is the solution with lowest positive angular momentum and smallest number of oscillations, this is so far an open question: see Appendix~\ref{appendix:open}. In Appendix~\ref{appendix:radial}, we also give numerical results that point in this direction.

\medskip We now focus on the dimension $d=1$. It turns out that one can completely solve~\eqref{eq:def:NLD_Intro} using special functions. Explicit formulae are given below, where $B$ and $_2F_1$ respectively denote the Euler \emph{Beta} function and the hypergeometric function.
\begin{theorem} \label{th:alphac_d=1}
Let $d=1$ and $p\in(1,+\infty)$. For all $\lambda\in[-\,m,m]$, the equation
\[
\Dirac \Psi-| \Psi |^\frac2{p-1}\,\Psi=\lambda\,\Psi\quad \mbox{with}\quad \Dirac :=\begin{pmatrix}
m&\partial_x\\-\,\partial_x&-\,m
\end{pmatrix}
\]
has a unique solution $\Psi\in\mathrm L^2(\R,\C^2)\setminus\{0\}$, up to a phase factor and a translation. Up to a translation, $V=|\Psi|^{2/(p-1)}$ is even, decreasing on $\R^+$ and such that $\alpha_D(\lambda,p)=\nrm Vp$.
\begin{itemize}
\item \emph{Subcritical regime $\lambda>-\,m$.} With $\mathsf A:=\frac p{p-1}\,\big(m^2-\lambda^2\big)$, $\mathsf B:=\frac 2{p-1}\,\sqrt{m^2-\lambda^2}$ and $z_0:=\frac{m-\lambda}{m+\lambda}$, we have
\be{eq:explicit_V_noncritical}
\forall\,x\in\R\,,\quad V(x)=\dfrac{\mathsf A}{m\,\cosh(\mathsf B\,x)+\lambda}
\ee
and $\big(\alpha_D(\lambda,p)\big)^p=p^p\,\big(\tfrac{m+\lambda}{p-1}\big)^{p-1}\,z_0^{p-1/2}\,B\(\tfrac12,p\)\,_2F_1\(\tfrac12,p;p+\tfrac12;-z_0\)$.
\item \emph{Critical case $\lambda=-\,m$.} With $\zeta=2\,m/(p-1)$, we have
\be{eq:explicit_V_critical}
\forall\,x\in\R\,,\quad V(x):=\frac{\zeta\,p}{1+\zeta^2\,x^2}
\ee
and $\big(\alpha_\star(p)\big)^p=p^p\,\big(\tfrac{2\,m}{p-1}\big)^{p-1}\,B\big(\tfrac12,p-\tfrac12\big)$.\end{itemize}
If $p=d=1$, then $\alpha_D(\lambda,1)=\arccos(\lambda/m)$ and
\[
\lim_{\lambda\to(-m)_+}\alpha_D(\lambda,1)=\lim_{p\to1_+}\alpha_\star(p)=\pi\,.
\]
\end{theorem}
\begin{figure}[ht]
\begin{subfigure}{0.55\textwidth}
\includegraphics[width=1\textwidth]{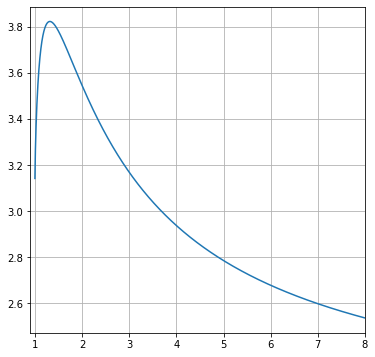}
\end{subfigure}
\begin{subfigure}{0.4\textwidth}
\includegraphics[width=1\textwidth]{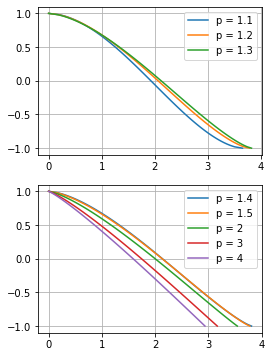}
\end{subfigure}
\caption{Let $d=1$ and $m=1$. (Left) The function $p\mapsto\alpha_\star(p)$, with maximum at $p \approx 1.32$, satisfies $\lim_{p \to1_+}\alpha_\star(p)=\pi$ and $\lim_{p\to+\infty}\alpha_\star(p)=2$. (Right) For various values of $p$, the maps $\alpha \mapsto \Lambda_D(\alpha, p)$ take value $-1$ at $\alpha=\alpha_\star(p)$. Upper (resp.~lower) right plots correspond to $p<1.32$ (resp.~$p>1.32$).}
\label{fig:alphac_d1}
\end{figure}
\noindent See Fig.~\ref{fig:alphac_d1}. With the notations of Theorem~\ref{Thm:Main1} and $\alpha=\alpha_D(\lambda,p)$, up to translations, we know that $V=V_{\alpha,p}$ in~\eqref{eq:explicit_V_noncritical} and~\eqref{eq:explicit_V_critical}. For the proof of Theorem~\ref{th:alphac_d=1} and some additional details, see Section~\ref{Sec:Explicit1}. Formally as $p\to1_+$, the potential given by~\eqref{eq:explicit_V_noncritical} converges to a \emph{delta} Dirac distribution at $x=0$ of mass $\arccos(\lambda/m)$ (see~\cite{MR1009526} for the study of self-adjoint extensions of $\Dirac - \alpha\,\delta_0$). A remarkable consequence of the estimate in the case $p=d=1$ is the Keller-Lieb-Thirring inequality
\be{cos}
m\,\cos\(\nrm V1\)\le\lambda_D(V)\le m
\ee
for any nonnegative potential $V\in\mathrm L^1(\R)$ with $\| V \|_{1} \le \pi$. See Appendix~\ref{App:D} for a result on optimality cases in the case $p=d=1$, with a proof.

The case of Theorem~\ref{th:alphac_d=1} presents some similarities with the results of~\cite{esteban2021diraccoulombII}: in the case $p=1$, it is expected that optimality is achieved only by singular measures. Our goals differ from those of~\cite{esteban2021diraccoulombII} as we adopt the point of view of functional interpolation inequalities with Keller-type estimates as a subproduct, while~\cite{esteban2021diraccoulombII} is concerned with the issue of the optimal charge distribution for a Dirac-Coulomb equation. In terms of methods, there are many similarities since we use Birman-Schwinger reformulations as well as classical tools of the concentration-compactness method. However, there are also significant differences because requesting that the potential is in $\mathrm L^p(\R^d)$ means that the optimal $V$ is obtained through a non-linear Dirac equation which is not measure-valued as soon as $p>1$.

\medskip Our results are not limited to estimates for the ground state and we also have a \emph{Lieb-Thirring inequality} for the sum of eigenvalues {\em in the gap} $(-\,m, m)$ of Dirac operators of the form $\Dirac - V$ with $V \in \mathrm L^p(\R^d, \R^+)$. We denote by $-\,m < \lambda_1\le \lambda_2 \le \cdots < m$ the possibly infinite sequence of eigenvalues in the gap $(-\,m,m)$, and write
\[
e_k = e_k(m, V) := (m - \lambda_k) > 0\,,
\]
so that $2\,m > e_1 \ge e_2 \ge \cdots > 0$. The quantity $e_k$ is the distance between the eigenvalue~$\lambda_k$ and the bottom of the upper essential spectrum $+\,m$.
\begin{theorem}\label{thm:LT-bound} For all $\gamma > d/2$ and $p \in (d, \gamma+d/2]$, there is a constant $L_{\gamma, d, p} > 0$ so that, for all $V \in \mathrm L^p(\R^d, \R^+)$, and all $m > 0$, we have
\be{eq:LT_RieszMean}
\sum_{k\ge1} e_k^\gamma(m,V) \le L_{\gamma, d, p}\,m^\frac d2\int_{\R^d} V_m^{\gamma + \frac d2 - p}\,V^{p}\,\rd x \quad \text{with} \quad V_m := \min \left\{ m, V \right\}\,.
\ee
\end{theorem}
If $V \in \mathrm L^p(\R^d, \R^+) \cap \mathrm L^{\gamma + d/2}(\R^d, \R^+)$, using the inequalities $V_m \le m$ and $V_m \le V$ gives respectively
\[
\sum_{k\ge1} e_k^\gamma \le L_{\gamma, d, p}\,m^{\gamma + d - p} \int_{\R^d}\,V^{p}\,\rd x \quad \text{and} \quad 
\sum_{k\ge1} e_k^\gamma \le L_{\gamma, d, p}\,m^\frac d2 \int_{\R^d} V^{\gamma + \frac{d}{2}}\,\rd x\,.
\]
The inequality~\eqref{eq:LT_RieszMean} is, in some sense, an interpolation between these two critical cases. In the proof, we use rough estimates: the method is constructive but there is a lot of space for improving on the constant $L_{\gamma, d, p}$.

\paragraph{Structure of the paper} This paper is organized as follows. In Section~\ref{Sec:Properties}, we establish some properties of the operator $\Dirac-V$ with $V\in\mathrm L^p(\R^d)$: domain, associated Birman–Schwin\-ger operator and self-adjointness. Section~\ref{sec:proof:alpha} is devoted to the variational problem associated with~\eqref{eq:min_vap_Dirac}, after reformulation in the Birman–Schwinger framework. Theorem~\ref{th:w} is devoted to the existence of an optimal potential $V$ by concen\-tra\-tion-compact\-ness methods (Section~\ref{sec:proof:w}). The regularity of the optimizers is studied in Section~\ref{Sec:Regularity}. Section~\ref{Sec:LT} is devoted to the proof of Theorem~\ref{thm:LT-bound}. Explicit and numerical computations are performed in Section~\ref{Sec:Explicit} in dimensions $d=1$, $2$ and $3$. Open questions, numerical observations, remarks on the non-relativistic limit and Gagliardo-Nirenberg-Sobolev inequalities, and a result in the case $p=d=1$ are collected in Appendices~\ref{appendix:open},~\ref{appendix:radial}, ~\ref{appendix:GNS} and~\ref{App:D} respectively.

\section{Properties of Dirac operators}\label{Sec:Properties}

\subsection{A self-adjoint realization}\label{sec:distinguished-extension}

We assume that $V\in\mathrm L^p(\R^d, \R^+)$ is positive valued and deal with the self-adjoint extensions of $\Dirac-V$. 
\begin{proposition} \label{prop:selfadjoint} Let $p\ge d\ge 1$ and $V\in\mathrm L^p(\R^d, \R^+)$. Then the operator $\Dirac-V$ is self-adjoint with domain:
\[
\textstyle\Dom(\Dirac-V):=\Big\{\psi\in\mathrm L^2(\R^d, \C^N):
\sqrt{V}\psi,\,(\Dirac-V)\,\psi\in\mathrm L^2(\R^d, \C^N) \Big\}\,.
\]
This is the unique self-adjoint realisation verifying: 
\[
\mathrm H^{1}(\R^d, \C^N)\subseteq\Dom(\Dirac-V)\subseteq \mathrm H^{\frac12}(\R^d, \C^N)\,.
\]
Moreover, we have the following properties.
\begin{enumerate}[label=\emph{(\roman*)}]
 \item\label{item:1prop:selfadjoint} If $p$ satisfies
\be{eq:cond_p_ess_sa}
\begin{cases}
p\ge 2&\quad \mbox{if}\quad d=1\,,\\
p > 2&\quad \mbox{if}\quad d=2\,,\\
p\ge d&\quad \mbox{if}\quad d\ge 3\,,
\end{cases}
\ee
then $\Dom(\Dirac-V)=\mathrm H^1(\R^d,\C^N)$.
\item\label{item:2prop:selfadjoint}
If $1 < p\le 2$ and $d=1$, then $\Dom(\Dirac-V)$ is also included in $\mathrm H^{\frac32-\frac1p}(\R,\C^2)$, hence in $\mathrm L^\infty(\R,\C^2)$.
\end{enumerate}
\end{proposition}
\noindent We call the extension of Proposition~\ref{prop:selfadjoint} the \emph{distinguished} extension, as it is the unique one whose domain is included in the formal form domain $\Dom(|\Dirac|^{1/2})=\mathrm H^{1/2}(\R^d,\C^N)$. We will consider only this extension in what follows, so that the operator \hbox{$\Dirac-V$} is self-adjoint under the condition~$p\ge d\ge 1$. The proof of the first part of Proposition~\ref{prop:selfadjoint} follows from~\cite{Nenciu_1976}. For completeness, we provide a short proof using the associated Birman-Schwinger operator. Under Condition~\eqref{eq:cond_p_ess_sa}, the point~\ref{item:1prop:selfadjoint} comes from the usual Kato-Rellich theorem~\cite{Rellich_1937, Kato_1951}. The result in~\ref{item:2prop:selfadjoint} is derived by bootstrapping the Sobolev embedding theorem. See Section~\ref{sec:proof:selfadjoint} for the proof of Proposition~\ref{prop:selfadjoint}.
\begin{remark} For comparison, it is interesting to consider limit cases. The Coulomb potential $V(x)=1/|x|$ in dimension $d=3$ is in the weak Sobolev space $\mathrm L^3_{w}(\R^d)$. The operator $\Dirac-\kappa\,V$ is essentially self-adjoint if $0\le \kappa\le \sqrt{3}/2$, it has a distinguished extension if $\sqrt{3}/2<\kappa\le 1$, and no distinguished self-adjoint extension if $\kappa >1$: see~\cite{MR3865751} and references therein. Also see Remark~\ref{Rem:Self-Ajointness}.\end{remark}

\subsection{The Birman-Schwinger operator}\label{Sec:Birman-SchwingerOperator}

The Birman-Schwinger operator is a powerful tool for analysing the spectral properties of $\Dirac-V$ when $V$ belongs to a large class of perturbations.
In the relativistic case, Klaus in~\cite{Klaus80} used it extensively to characterize and study the first eigenvalue of Dirac operators when proving the existence of a distinguished self-adjoint extension. For non-Hermitian potentials $V$, it can be employed to locate the eigenvalues of $\Dirac-V$, as shown for example by Cuenin, Laptev and Tretter in~\cite{CueLapTre2014}, and by Fanelli and Krej\v{c}i\v{r}\'{\i}k in~\cite{FanKrej19}.
Furthermore, it can be applied to discuss properties of the ground state of $\Dirac-V$ when~$V$ is a generalised Coulomb-type potential, see, \emph{e.g.},~\cite{Klaus80, CPV20, ELS2021, esteban2021diraccoulombII}. 
Throughout this paper, following the approach by Kato~\cite{Kato_1983} and by Konno and Kuroda~\cite{Konno-Kuroda}, the Birman-Schwinger operator is used to define the self-adjoint extension of the operator $\Dirac-V$.
Then, with this rigorous definition at hand, we prove the existence of the optimization problem which defines the \emph{ground state} by applying variational methods directly on the Birman-Schwinger reformulation of the problem.

For $z \notin \sigma(\Dirac)$, let
\be{eq:def:Rz}
R_0(z) :=(\Dirac-z)^{-1}
\ee
denote the resolvent operator. Recall that we assume $V\ge 0$. We introduce the \emph{Birman-Schwinger operator}
\be{eq:def:KV}
K_V(z) :=\sqrt V\,R_0(z)\,\sqrt V=\sqrt V\,\dfrac1{\Dirac-z}\,\sqrt V\,.
\ee
\begin{lemma} \label{lem:KV}
For all $p\ge d\ge 1$, all $V\in\mathrm L^p(\R^d, \R^+)$ and all $z \notin \sigma(\Dirac)$, the operator $K_V(z)$ is compact (hence bounded). In addition,
\[
\lim_{s \to\pm \infty}\left\|K_V( \ri\,s ) \right\|_{\rm op}=0\,.
\]
\end{lemma}
This result follows from the \emph{Kato-Seiler-Simon inequality}: see proof in Section~\ref{sec:proof:selfadjoint}. A consequence of Lemma~\ref{lem:KV} is the following result (also see Section~\ref{sec:proof:selfadjoint} for the proof).
\begin{proposition}\label{Prop:Birman-Schwinger}
Let $\Dirac-V$ be the distinguished self-adjoint extension defined as in Proposition~\ref{prop:selfadjoint}. Then
\[
\sigma_{\rm ess}(\Dirac-V)=\sigma_{\rm ess}(\Dirac)=(-\infty,-m] \cup [m,+\infty)\,.
\]
Moreover the \emph{Birman-Schwinger principle} holds: for all $\lambda\in (-\,m,m)$, $\lambda$ is an eigenvalue of $\Dirac-V$ if and only if $1$ is an eigenvalue of $K_V(\lambda)$.
\end{proposition}

Let us point out some differences with Birman-Schwinger operators associated with Schr\"odinger operators (see Figure~\ref{fig:spectrum_BS}). 
In the Schrödinger case, the Birman-Schwinger operator is of the form
\[
\widetilde{K_V}(\lambda)=\sqrt V\,\dfrac1{-\Delta-\lambda}\,\sqrt V\,.
\]
\begin{figure}[ht]
\begin{subfigure}{0.45\textwidth}
\includegraphics[width=1\textwidth]{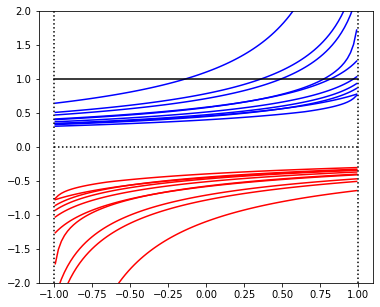}
\end{subfigure}
\begin{subfigure}{0.45\textwidth}
\includegraphics[width=1\textwidth]{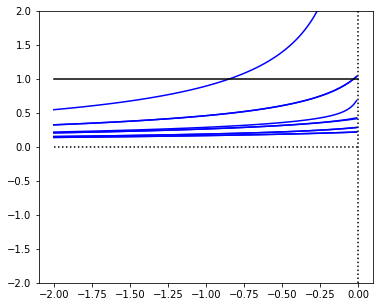}
\end{subfigure}
\caption{\textit{\textbf{Birman-Schwinger principle}}. (Left). The spectrum of $\lambda \mapsto K_V(\lambda)$ (Dirac case) for $\lambda \in (-1, 1)$, and $V(x) = 2\exp(-| x |^2/4)$ in dimension $d = 2$. We only plotted the 10 largest (blue) and 10 lowest (red) eigenvalues. An energy $\lambda$ is an eigenvalue of $\Dirac - V$ if one eigenvalue of $K_V(\lambda)$ crosses the black line~$1$. (Right) Same for $\lambda \mapsto \widetilde{K_V}(\lambda)$ (Schrödinger case) with $\lambda \in (-2, 0)$.}
\label{fig:spectrum_BS}
\end{figure}

For any $\lambda < 0$, the operator $\widetilde{K_V}(\lambda)$ is a positive compact operator and the map $\lambda \mapsto \widetilde{K_V}(\lambda)$ is operator increasing on $\R^-$. In particular, if $\widetilde{\mu_1}(\lambda) > \widetilde{\mu_2}(\lambda)\ge \cdots\ge 0$ denote the eigenvalues of $\widetilde{K_V}(\lambda)$, ranked in decreasing order and counted with multiplicities, all functions $\lambda \mapsto \widetilde{\mu_j}(\lambda)$ are increasing on $\R^-$. In addition, the first eigenvalue $\widetilde{\mu_1}$ is simple because the kernel $\widetilde{K_V}(x,y)$ is pointwise positive, together with Krein-Rutman theorem: see~\cite[Theorem 6.13]{Brezis_2011} for a statement and also~\cite[Section~XIII.12]{reedsimon4}.

In the Dirac case, the operator $K_V(\lambda)$ with $\lambda\in\R$ is defined only in the gap $(-\,m, m)$ of the essential spectrum. It is compact by Lemma~\ref{lem:KV} and symmetric because $\lambda$ is real, but it is \emph{not} a positive operator. Its eigenvalues are real valued, and can be ranked as $\mu_1(\lambda)\ge \mu_2(\lambda)\ge \cdots\ge 0$ for the positive eigenvalues, and $\nu_1(\lambda)\le \nu_2(\lambda)\le \cdots\le 0$ for the negative ones. As the map $\lambda \mapsto (\Dirac-\lambda)^{-1}$ is operator increasing on $(-\,m, m)\ni\lambda$, all maps $\lambda \mapsto \mu_j(\lambda)$ and $\lambda \mapsto \nu_j(\lambda)$ are increasing. This explains in particular why we expect eigenvalues to emerge from the upper essential spectrum in this setting. We do not know whether $\mu_1(\lambda)$ is always a simple eigenvalue or not (see Appendix~\ref{appendix:open} for more details on open questions).

\medskip For $\lambda\in (-\,m, m)$, $p\ge d\ge 1$, and $V\in\mathrm L^p(\R^d, \R^+)$, let $\mu_1\big(K_V(\lambda)\big)$ denote the largest (positive) eigenvalue of $K_V(\lambda)$. We rephrase the optimization problem~\eqref{eq:def:alpha0} as
\be{eq:def:alpha}
\alpha_D(\lambda,p):=\inf\Big\{\nrm Vp\,:\,V\in\mathrm L^p(\R^d,\R^+)\;\mbox{and}\;\mu_1\big(K_V(\lambda)\big)=1\Big\}\,.
\ee

\subsection{Proofs of \texorpdfstring{Proposition~\ref{prop:selfadjoint}, Lemma~\ref{lem:KV} and Proposition~\ref{Prop:Birman-Schwinger}}{Proposition 2.1, Lemma 2.3 and Propositions 2.4}}\label{sec:proof:selfadjoint}

We start by establishing that $K_V$ defined by~\eqref{eq:def:KV} is a compact operator (Lemma~\ref{lem:KV}) before proving Propositions~\ref{prop:selfadjoint} and~\ref{Prop:Birman-Schwinger}.

\begin{proof}[Proof of Lemma~\ref{lem:KV}]
Assume first that $p > d$. We claim that, for $z \notin \sigma(\Dirac)$, the operator $K_V(z)$ is compact. We have $R_0(z)=g_z(-\ri\,\nabla )$ with
\[
g_z(k) :=\dfrac1{|k|^2+m^2-z^2} \(\sum_{j=1}^d \alpha_j\,k_j+m\,\beta+z\,\bbI_N\)=: \sum_{j=1}^d g_z^j(k)+g_z^m(k)+g_z^z(k)\,,
\]
with obvious notation. Let us focus on the $g_z^1(k)$ term. We write $g_z^1(k)=g_z^A(k)\,g_z^B(k)$ with
\[
g_z^A(k) :=\dfrac{\sqrt{| k_1 |}\,{\rm sgn}(k_1)}{\sqrt{|k|^2+m^2-z^2}}\,\alpha_1\quad \mbox{and}\quad
g_z^B(k) :=\dfrac{\sqrt{| k_1 |}}{\sqrt{|k|^2+m^2-z^2}}\,\mathbb{I}_d\,.
\]
All components of the functions $g_z^A$ and $g_z^B$ are in $\mathrm L^q(\R^d)$ for all $q > 2\,d$. Since $\sqrt V$ is in $\mathrm L^{2p}(\R^d)$ with $2\,p > 2\,d$, we can use the Kato-Seiler-Simon inequality~\cite[Chapter 4, Theorem~4.1]{SimonTraceIdeals} and conclude that the operator
\[
K_V^1(z) :=\sqrt V\,g_z^1(-\ri\,\nabla )\,\sqrt V
\]
is in the Schatten class $\fS_p\big(\mathrm L^2(\R^d,\C^N)\big)$ with
\[
\|K_V^1(z)\|_{\fS_p}\le\left\|\sqrt{V(x)}\,g_z^A(-\ri\,\nabla) \right\|_{\fS_{2p}}\,\left\|g_z^B(-\ri\,\nabla)\,\sqrt{V(x)}\,\right\|_{\fS_{2p}}\le C\,\|V\|_p\,\|g_z^1\|_p\,.
\]
In addition, we have $\|g_{z=is}^1\|_p \to 0$ as $s \to\pm\infty$. Similar computation for the other terms shows that $K_V$ is in the Schatten class $\fS_p$, and that $\lim_{s \to\pm \infty}\|K_V(\ri\,s)\|_{\rm op}=0$.

\medskip In the case $p=d$ with $d\ge 2$, we use that all components of the functions $g_z^A$ and $g_z^B$ are in the weak-Sobolev space $\mathrm L^{2\,d}_w(\R^d)$. According to~\cite[Chapter 4]{SimonTraceIdeals}, $g_z^A(-\ri\,\nabla)\,\sqrt V$ and $\sqrt V\,g_z^B(-\ri\,\nabla)$ are in the weak Schatten class $\fS_{2d, w}$. In particular, they are compact operators. This already proves that $K_V$ is compact as well. Note that $\|g_{z=\ri\,s}^1\|_{d,w}$ does not converge to $0$ as $s \to\pm\infty$.

For any $R > 0$, $\sqrt{V_R} :=\min\big(\sqrt V, R\big)$ belongs to $\mathrm L^d(\R^d) \cap \mathrm L^\infty(\R^d)$. We have
\[
\left\|g_z^A(-\ri\,\nabla)\,\sqrt V \right\|_{\rm op}\le\left\|g_z^A(-\ri\,\nabla)\,\sqrt{V_R} \right\|_{\rm op}+\left\|g_z^A(-\ri\,\nabla)\,\big(\sqrt V-\sqrt{V_R} \big) \right\|_{\rm op}\,.
\]
For $R$ large enough, the second term is small in the Schatten space $\fS_{2d, w}$, and for $z=\ri\,s$ with $|s|$ large, the first term is small in $\fS_q$ with $q > p$. Hence \hbox{$\lim_{s \to\pm\infty}\|K_V(\ri\,s)\|_{\rm op}=0$}.

\medskip Let us finally assume $p=d=1$. In this case, with explicit computations, the kernel of the Birman-Schwinger operator $K_V(z)$ is given by 
\[
\sqrt{V(x)}\cdot
\frac{1}{2}
\begin{pmatrix}
\frac{z+m}{k}&\operatorname{sign}(x-y)\\
\operatorname{sign}(x-y) &\frac{z-m}{k}
\end{pmatrix}\re^{-k\,|x-y|}\cdot
\sqrt{V(y)}
\in \mathrm L^2(\mathbb{R}\times\mathbb\R,\mathbb{C}^2)
\]
where $k=\sqrt{m^2-z^2}$ is chosen with a positive real part. Thus, $K_V$ is a Hilbert-Schmidt operator (hence it is compact) and by the dominated convergence theorem we can conclude that
\hbox{$\lim_{s \to\pm\infty}\|K_V(\ri\,s)\|_{\rm op}=0$}.
\end{proof}

\begin{proof}[Proof of Proposition~\ref{prop:selfadjoint}] We divide the proof in several steps.

\medskip\noindent\emph{\textbf{Distinguished self-adjoint extension.}}
We define the domain of self-adjointness for the operator $\Dirac-V$ as a perturbation of $\Dirac$ by applying the method of G.~Nenciu in~\cite{Nenciu_1976}.
Using similar techniques as in the proof of Lemma~\ref{lem:KV}, one can show that the operators $R_0(z)\,\sqrt V$ and $\sqrt V\,R_0(z)$ can be extended into bounded linear operators on $\mathrm L^2(\R^d, \C^N)$. These operators are compact operators, in the Schatten class $\fS_{2p}$. We are now in the setting of~\cite{Nenciu_1976}. Let $\Omega :=\{z\in\C\,:\,1 \notin \sigma\big(K_V(z)\big) \}$ where $K_V$ is defined by~\eqref{eq:def:KV}. The set
~$\Omega$ is non-empty by Lemma~\ref{lem:KV}. For $z\in\Omega$, define
\[
R(z) :=R_0(z)+R_0(z)\,\sqrt V\,\big(1-K_V(z)\big)^{-1}\,\sqrt V\,R_0(z)\,.
\]
According to~\cite{Nenciu_1976}, the operator $\Dirac-V$ has a unique self-adjoint extension whose resol\-vent is the operator $R(z)$ defined in~\eqref{eq:def:Rz}. Its domain is $\Dom(\Dirac-V) :={\rm Ran}\,R(z)$, which is independent of $z\in\Omega$. This is the unique extension which is included in the formal form domain $\Dom(|\Dirac|^{1/2})=\mathrm H^{1/2}(\R^d, \C^N)$.

\medskip\noindent\emph{\textbf{Domain of the distinguished extension.}}
Define the \emph{maximal domain} as
\[
\Dom_{\rm max}(\Dirac-V):=\big\{\psi\in\mathrm L^2(\R^d, \C^N)\,:\,(\Dirac -V)\,\psi \in\mathrm L^2(\R^d, \C^N)\big\}.
\]
Then, the set
\[
\left\{\psi\in\mathrm L^2(\R^d, \C^N)\,:\, \sqrt V\,\psi,\;(\Dirac-V)\psi\in\mathrm L^2(\R^d, \C^N) \right\}
\]
is also $\Dom_{\rm max}(\Dirac-V) \cap \Dom\big(\sqrt V\big)$. We write $\psi\in\Dom (\Dirac-V)$ as $R(z)\,f$ for some $f\in\mathrm L^2(\R^d,\C^N)$.
Then
\[
\sqrt V\,\psi=\sqrt V\,R_0(z)\,f+K_V\,\big(1-K_V(z)\big)^{-1} \sqrt V\,R_0(z)\,f\in\mathrm L^2(\R^d,\C^N)\,.
\]
This proves that $\Dom(\Dirac-V) \subset \Dom_{\rm max}(\Dirac-V) \cap \Dom \big(\sqrt V\big)$. For the opposite inclusion, consider $\psi\in\Dom_{\rm max}(\Dirac-V) \cap \Dom\big(\sqrt V\big)$. We set $ f:=(\Dirac+V-z)\,\psi\in\mathrm L^2(\R^d, \C^N)$ and $\psi_0=R(z)\,f\in\Dom(\Dirac-V)$. Note that $\sqrt V\,\psi_0\in\mathrm L^2(\R^d, \C^N)$, since
\[
\sqrt V\,R(z)=\sqrt V\,R_0(z)+K_V(z)\,\big( 1-K_V(z)\big)^{-1}\,\sqrt V\,R_0(z)
\]
is a bounded operator on $\mathrm L^2(\R^d, \C^N)$. So $\phi:=\psi-\psi_0$ is such that
\[
\big(\Dirac-V-z\big)\,\phi=0\quad\mbox{and}\quad \phi\in\Dom\big(\sqrt V\big)\,.
\]
{}From the relation
\[
\big(1-R_0(z)\big)^{-1}\,\sqrt V\,R_0(z)\,\big(\Dirac-V-z\big)=\sqrt V\,,
\]
we obtain $\sqrt V\,\phi=0$, and from the relation
\[
R_0(z)\,\big(\Dirac-V-z\big)=1-R_0(z)\,V=1-\big( R_0(z)\,\sqrt V\,\big)\,\sqrt V\,,
\]
we finally get $\phi=0$, hence $\psi=\psi_0\in\Dom(\Dirac-V)$.

Finally, if $p\ge d\ge 1$ by the H\"older inequality and the Sobolev embedding theorem we get that $\mathrm H^1(\R^d,\C^N) \subseteq \Dom(\Dirac-V)$ and this concludes the first part of the proof.

\medskip\noindent\emph{\textbf{Self-adjointness on $\mathrm H^1(\R^d,\C^N)$.}}
Let us prove~\ref{item:1prop:selfadjoint}.
Assume that $p$ satisfies~\eqref{eq:cond_p_ess_sa}. Thanks to~\eqref{eq:def:alphai_beta} we have
\[
\forall\,\psi\in\mathrm H^2(\R^d,\C^N)\,,\quad\left\|\Dirac \psi\right\|_2^2=\|\nabla \psi\|_2^2+m^2\,\|\psi\|_2^2\,.
\]
This shows that the graph norm of $\Dirac$ is equivalent to the usual $\mathrm H^1(\R^d,\C^N)$ norm. Set
\[
q :=\begin{cases}
\frac{2\,p}{p-2}&\quad \mbox{if}\quad p > 2\\
+\infty&\quad \mbox{if $d=1$ and $p=2$}
\end{cases}\quad \mbox{so that}\quad \frac1p+\frac1q=\frac12\,.
\]
We write $V=V_1+V_2$ with $V_1 :=V\,\1_{V\ge R}$ and $V_2 :=V\,\1_{V\le R}$. We have
\begin{align*}
\|V\,\psi\|_2&\le\|V_1\,\psi\|_2+\|V_2\,\psi\|_2\le\|V_1\|_p\,\|\psi\|_q+\|V_2\|_\infty\,\|\psi\|_2\\
&\le C_S\,\|V_1\|_p\,\|\psi\|_{\mathrm H^1}+\|V_2\|_\infty\,\|\psi\|_2
\end{align*}
where, in the last inequality, we used Sobolev's embedding $\mathrm H^1 (\R^d) \hookrightarrow \mathrm L^q(\R^d)$ and, according to~\eqref{eq:cond_p_ess_sa}, $q$ satisfies
\[
\begin{cases}
2\le q\le+\infty&\quad \mbox{if}\quad d=1\,,\\
2\le q <+\infty&\quad \mbox{if}\quad d=2\,,\\
2\le q\le \frac{2\,d}{d-2}&\quad \mbox{if}\quad d\ge 3\,.
\end{cases}
\]
We choose $R$ large enough so that $C_S\,\|V_1\|_p < 1$ and conclude with the \emph{Kato-Rellich theorem} (see~\cite[Theorem~X.12]{MR0493420}) that $\Dirac-V$ is self-adjoint with domain $\mathrm H^1(\R^d,\C^N)$.
Since any self-adjoint operator only admits trivial self-adjoint extensions, we can conclude that $\mathrm H^1(\R^d,\C^N)= \Dom(\Dirac-V)$.

\medskip\noindent\emph{\textbf{Regularity for $d=1$.}}
Let us focus on~\ref{item:2prop:selfadjoint} and assume that $d=1$ and $1 < p\le 2$. Let us prove that $\Dom(\Dirac-V)$ is also included in $\mathrm H^{\frac32-\frac1p}(\R,\C^2)$. For any $\psi\in\Dom(\Dirac-V)$, we have
\[
(\Dirac-V)\,\psi=: f\,\in\mathrm L^2(\R,\C^2)\,,
\]
hence $\Dirac\,\psi=f+V \psi$. We recall the following negative Sobolev embeddings: for all $1 < r\le 2$, we have $\mathrm L^r(\R) \hookrightarrow \mathrm H^{-s}(\R)$ for all $s\ge \frac{2-r}{2\,r}$ and $\mathrm L^2(\R) \hookrightarrow \mathrm H^{-s}(\R)$ for all $s\ge 0$, while
\[
\forall\,s\ge \tfrac1{2\,p}\,,\quad V\,\psi=\sqrt V\,\underbrace{\big(\sqrt V\,\psi\big)}_{\in\mathrm L^2(\R,\C^2)}\in\mathrm L^{\frac{2\,p}{p+1}}(\R,\C^2) \hookrightarrow \mathrm H^{-s}(\R,\C^2)\,.
\]
We deduce that $\Dirac \psi\in\mathrm H^{-\frac1{2\,p}}(\R,\C^2)$, hence that $\psi\in\mathrm H^{1-\frac1{2\,p}}(\R,\C^2)$. We now bootstrap the argument. For $ p > 1$, we have $1-\frac1{2\,p} > \frac12$, so $\psi\in\mathrm L^\infty(\R,\C^2)$ by Sobolev's embedding. This gives $V\,\psi\in\mathrm L^p(\R,\C^2) \hookrightarrow \mathrm H^{\frac{p-2}{2\,p}}(\R,\C^2)$. So $\Dirac \psi=f-V\,\psi\in\mathrm H^{\frac{p-2}{2\,p}}(\R,\C^2)$ as well, and we obtain $\psi\in\mathrm H^{1+\frac{p-2}{2\,p}}(\R,\C^2)$ with $1+\frac{p-2}{2\,p}=\frac32-\frac1p$, as wanted.
\end{proof}

\begin{proof}[Proof of Proposition~\ref{Prop:Birman-Schwinger}] Since $R_0(z)\,\sqrt V$ is compact, then $R(z)$ is a compact perturbation of the free resolvent $R_0(z)$. 
The result on $\sigma_{\rm ess}(\Dirac-V)$ follows from~\cite[Theorem~4.5]{MR1219537} (also see~\cite[Theorem~XIII.14 and Corollary~1]{reedsimon4}). 
Such a result is known in the literature as \emph{Weyl's theorem}. 

Moreover, by construction, the Birman-Schwinger principle holds for the distinguished self-adjoint extension defined as in Proposition~\ref{prop:selfadjoint}: ${\lambda\in (-\,m,m)}$ is an eigenvalue of \hbox{$\Dirac-V$} if and only if $1$ is an eigenvalue of $K_V(\lambda)$.
See~\cite[Theorem~1.3]{MR985401} for a similar application of the Birman-Schwinger principle in a non-relativistic setting.\end{proof}

\begin{remark}\label{Rem:Self-Ajointness} The self-adjointness of Dirac operators involving potentials with one Coulomb singularity or several Coulomb singularities has been intensively studied in respectively~\cite{MR0355655,MR0326448,MR0365233,Nenciu_1976,MR0437948,Kato_1983,MR3088216} (with additional references therein) and~\cite{MR0462346,Klaus80}. In the alternative strategy of~\cite{MR2370231,MR2466677} based on~\cite{MR1761368}, a distinguished self-adjoint extension is built using the underlying Hardy inequality, which was related with the other constructions for Dirac-Coulomb operators in~\cite{MR3960263,ELS2021}. Also see~\cite{SchSolTok20,dolbeault2023distinguished} for further considerations on min-max principles, Hardy inequalities and self-adjointness issues. Optimal Hardy inequalities have been repeatedly use to establish optimal conditions for the existence of a ground state. For instance, in presence of a magnetic field as in~\cite{1751-8121-41-18-185303,exner2021partial,MR2333781}, a critical magnetic field is obtained as the ground state energy approaches $-\,m\,c^2$, which determines the optimal constant of the corresponding Hardy inequality. In the approach of~\cite{ELS2021,esteban2021diraccoulombII} as well as in our paper, the Birman-Schwinger formula is essential as it was in~\cite{Klaus80,MR0552332,Nenciu_1976}. Notice that we do not rely on Nenciu’s method~\cite[Corollary~2.1]{Nenciu_1976}, but instead use the method of Konno and Kuroda~\cite{Konno-Kuroda} and Kato's approach~\cite{Kato_1983}.\end{remark}

\section{The variational problem}\label{sec:proof:alpha}

In this section, we consider the minimization problem~\eqref{eq:def:alpha} and prove Theorem~\ref{Thm:Main1} in a reformulation which relies on the Birman-Schwinger operator associated to $\Dirac-V$, as introduced in Section~\ref{Sec:Birman-SchwingerOperator}. The proof of Theorem~\ref{Thm:Main1} is given below, right after the statement of Corollary~\ref{Cor:Main1}, as a simple consequence of previous results in the Birman-Schwinger framework.

\subsection{An auxiliary maximization problem}
First, we notice that, for all $t > 0$, we have $K_{tV}(\lambda)=t\,K_V(\lambda)$, hence $\mu_1\big( K_{tV}(\lambda) \big)=t\,\mu_1 \big( K_V(\lambda) \big)$. So, introducing the auxiliary problem
\be{eq:def:N}
\cN(\lambda,p) :=\sup \Big\{\mu_1 \big(K_W(\lambda)\big)\,:\,W\in\mathrm L^p(\R^d, \R^+)\,,\;\|W\|_p=1 \Big\}\,,
\ee
we deduce that
\be{Eq:alphaN}
\alpha_D(\lambda,p)=\frac1{\cN(\lambda,p)}\,.
\ee
If $W$ is a maximizer for $\cN(\lambda,p)$, then $V=W / \cN(\lambda,p)$ is a minimizer for $\alpha_D(\lambda,p)$.
In what follows, we study the maximization problem~\eqref{eq:def:N}. We perform several changes of variables to study this problem. First, the min-max principle shows that $\cN(\lambda,p)$ equals
\[
\cN(\lambda,p)=\sup_{\begin{array}{c}W\in\mathrm L^p(\R^d, \R^+)\\ \|W\|_p=1\end{array}} \sup_{\begin{array}{c}\phi\in\mathrm L^2(\R^d,\C^N)\\ \|\phi\|_2=1\end{array}}\left\langle \phi, \sqrt W\,R_0(\lambda)\,\sqrt W \phi \right\rangle\,.
\]
We make the change of variable
\[
w :=\sqrt W \phi
\]
so that, by H\"older's inequality, $w\in\mathrm L^{\frac{2\,p}{p+1}}(\R^d,\C^N)$, and, with the convention that $\nrm{w}{r}= \nrm{|w|_{\C^N}}{r} $,
\[
\left\|w \right\|_\frac{2\,p}{p+1}\le\|W\|_p^\frac12\,\|\phi\|_2=1\,.
\]
In addition, there is equality if and only if $\mathrm W^p$ is proportional to $| \phi |^2$, both proportional to $|w|^{\frac{2\,p}{p+1}}$. With
\[\label{Eq:pq}
q :=\frac{2\,p}{p+1}\in (1,2)\,,
\]
this shows that $\cN(\lambda,p)$ is also solution to the optimization problem
\be{eq:Nbis}
\cN(\lambda,p)=\sup \Big\{\big\langle w, R_0(\lambda)\,w \big\rangle\,:\,w\in\mathrm L^q(\R^d,\C^N)\,,\;\|w\|_q=1 \Big\}\,.
\ee
In addition, if $w\in\mathrm L^q(\R^d, \R^+)$ is an optimizer of~\eqref{eq:Nbis}, then the corresponding optimal~$W$ and $\phi$ are given by
\[
W=|w|^\frac qp=|w|^{\frac2{p+1}}\quad \mbox{and}\quad \phi=|w|^{\frac q2-1}\,w=|w|^{-\frac1{p+1}}\,w\,.
\]
Thus, by showing the existence of an optimizer for~\eqref{eq:Nbis}, we solve problem~\eqref{eq:def:N}, and by definition of the Birman-Schwinger operator, find an optimal potential and eigenfunction for our original problem~\eqref{eq:def:alpha0}.

Since $\alpha \mapsto \Lambda_D(\alpha, p)$ is the inverse map of $\lambda \mapsto \alpha_D(\lambda, p)$ according to~\eqref{Eq:alphaN}, and since $\alpha_D(\lambda, p) = 1/\cN(\lambda, p)$, it is enough to focus on the properties of $\cN(\cdot, p)$. 
\begin{theorem} \label{th:w} Let us consider $\cN$ defined by~\eqref{eq:Nbis}.
For all $\lambda\in (-\,m, m)$ and all $p > d$, we have $\cN(\lambda,p) > 0$. All maximizing sequences for~\eqref{eq:Nbis} are precompact up to translations, hence~\eqref{eq:Nbis} has maximizers. If $w$ is such an optimizer, then $w$ satisfies the Euler-Lagrange equation
\be{eq:EL_for_w}
R_0(\lambda)\,w=\tau\,|w|^{-\frac2{p+1}} w\quad \mbox{with}\quad \tau=\cN(\lambda,p)\,.
\ee
Finally, the map $\lambda \mapsto \cN(\lambda,p)$ is continuous, strictly increasing, and satisfies 
\[
 \lim_{\lambda \to -m}\cN(\lambda, p) =: \cN_c(p) > 0 \quad \text{and} \quad \lim_{\lambda \to +m}\cN(\lambda, p) = \infty.
\]
\end{theorem}
The proof of the first part relies on the profile decomposition method (concentration-compactness) used by Lions~\cite{Lions-84}, and is given in the next section. Theorem~\ref{th:w} implies the existence of an optimal potential and an optimal spinor. 
\begin{corollary}\label{Cor:Main1} 
Under the assumptions of Theorem~\ref{Thm:Main1}, the infimum~\eqref{eq:def:alpha} is attained for any $\lambda\in(-\,m,m)$ by a potential $V=|\Psi|^{2/(p-1)}$, where $\Psi\in\mathrm L^2(\R^d,\R^N)$ solves the \emph{nonlinear Dirac equation}
\be{eq:def:KellerDirac}
\Dirac \Psi-| \Psi |^{\frac2{p-1}}\,\Psi=\lambda\,\Psi\,,
\ee
such that $\lambda_D(V)=\lambda$ and $\big(\ird{|\Psi|^{2\,p/(p-1)}}\big)^{1/p}=\nrm Vp=\alpha_D(\lambda,p)=1/\cN(\lambda,p)$.
\end{corollary}
\begin{proof}[Proof of Theorem~\ref{Thm:Main1}]
Our main Theorem~\ref{Thm:Main1} is a direct consequence of Theorem~\ref{th:w} and Corollary~\ref{Cor:Main1}. Since $\cN_c(p) > 0$, we have indeed $\alpha_c(p) := 1/\cN_c(p) < \infty$.
\end{proof}

\begin{proof}[Proof of Corollary~\ref{Cor:Main1}]
First, we translate the Euler-Lagrange equation for $w$ into an equation for the potential $V$ and an eigenfunction (not normalized) $\Psi$. We set
\[
\Psi=\tau^{\frac{1-p}2}\,|w|^{-\frac2{p+1}}\,w\quad \mbox{so that}\quad
w=\tau^{\frac{p+1}2}\,| \Psi |^{\frac2{p-1}}\,\Psi\,.
\]
Applying $\Dirac-\lambda$ to~\eqref{eq:EL_for_w} shows that $\Psi$ satisfies the nonlinear Dirac equation~\eqref{eq:def:KellerDirac}. The optimal potential $W$ for the $\cN(\lambda,p)$ problem in~\eqref{eq:def:N} is $W=|w|^{\frac2{p+1}}=\tau\,| \Psi |^{\frac2{p-1}}$, and finally, the optimal potential $V$ for the $\alpha_D(\lambda,p)$ problem is, as wanted,
\[
V=\frac{W}{\cN(\lambda,p)}=| \Psi |^{\frac2{p-1}}\,.
\]
We recover the value of $\cN(\lambda,p)$ and $\alpha_D(\lambda,p)$ from the solution~$\Psi$ because
\[
\ird{ | \Psi |^\frac{2\,p}{p-1}}=\tau^{-p} \ird{ |w|^\frac{2\,p}{p+1}}=\tau^{-p}=\cN(\lambda,p)^{-p}=\alpha_D(\lambda,p)^p\,.
\]
Among all solutions of~\eqref{eq:def:KellerDirac}, $\Psi$ is the one with the smallest $\mathrm L^{\frac{2\,p}{p-1}}(\R^d,\C^N)$ norm so that $\lambda=\Lambda_D(\alpha,p)$ and $\Psi$ actually solves~\eqref{eq:def:NLD_Intro}. 
\end{proof}

\subsection{Proof of Theorem~\ref{th:w}}\label{sec:proof:w}

We now prove Theorem~\ref{th:w}. We consider a more general case, and study a general optimization problem.
In what follows, we use the notation
\[
\big\langle w,K * w\big\rangle :=\iint_{\R^d\times\R^d} \big\langle w(x),K(x-y)\,w(y)\big\rangle_{\C^N}\,\rd x\,\rd y
\]
and define for any $s>0$ the maximization problem
\be{eq:J}
J(s) :=\sup\left\{\big\langle w, K * w \big\rangle\,:\,w\in\mathrm L^q(\R^d,\C^N)\,, \; \ird{ |w|^q}=s\right\}\,.
\ee
Here, $K$ is a convolution operator, or equivalently a multiplication operator in Fourier space. In our case, $K(x-y)=R_0(\lambda)(x-y)$ is the kernel of the Dirac resolvent, but we state a more general result.
\begin{lemma} \label{lem:key_lemma}
Let $q\in(1,2)$, set $q' :=q/(q-1)\in (2,+\infty)$ and $r :=q'/2\in (1,+\infty)$. Let $K : \R^d \to \cM_N(\C)$ be a matrix-valued function satisfying $K(x)=K(-x)^*$, and such that one of the two properties holds:
\begin{enumerate}[label=\emph{(\roman*)}]
\item\label{item:1-key_lemma} either $K\in\mathrm L^r(\R^d, \cM_N(\C))$,
\item\label{item:2-key_lemma} or $K=R_0(\lambda)$ is a Dirac resolvent for some $\lambda\in(-\,m,m)$.
\end{enumerate}
Then the map $w \mapsto \big\langle w, K * w \big\rangle$ is well-defined on $\mathrm L^q(\R^d,\C^N)$ and real valued. 
Moreover, if $J(1) > 0$, then~\eqref{eq:J} admits maximizers.
\end{lemma}
Before proving this result, we make several remarks.
\begin{remark}
Lemma~\ref{lem:key_lemma} fails at the endpoint $q=2$. Indeed, by applying the Fourier transform we have
\[
\big\langle w, K * w \big\rangle=\int_{\R^d}\left\langle \widehat{w}(k), \widehat{K}(k)\,\widehat{w}(k) \right\rangle_{\C^N}\,\rd k\,.
\]
This means that all optimizing sequences must concentrate on Dirac masses in Fourier space at locations where $k \mapsto \sup \operatorname{spec} \big(\widehat{K}(k)\big)$ has maxima. Since the Fourier transform is an isometry on $\mathrm L^2(\R^d)$, we deduce that the maximization problem has no maximum in general. The same argument shows that the existence of optimizers is closely related to the fact that the Fourier transform is not a bijection between $\mathrm L^q(\R^d)$ and $\mathrm L^{q'}(\R^d)$ if $1 < q < 2$.
\end{remark}
\begin{remark} \label{rmk:lemma_applies}
In the case of the Dirac operator, one has an explicit expression for $K=R_0(\lambda)$, the fundamental solution of $\Dirac-\lambda$. Using that
\[
\(\Dirac-\lambda\)^{-1}=\(\Dirac+\lambda\) \dfrac1{-\Delta+m^2-\lambda^2}\,,
\]
we first deduce that $R_0(\lambda)(\cdot)$ is the Fourier transform of
\[
g_\lambda(k)=\(\sum_{j=1}^d \alpha_j\,k_j+m\,\beta+\lambda\,\mathbb I_N\) \frac1{k^2+m^2-\lambda^2}\,.
\]
The function $k \mapsto g_\lambda(k)$ is analytic on $\R^d$ because there is no singularity in the denominator since $| \lambda | < m$, so its Fourier transform is exponentially decaying in $x$. Actually, we have
\begin{multline*}
R_0(\lambda)(x)=\frac{c_{d,\lambda}}{|x|^{d/2-1}}
\Bigg(\ri\sum_{j=1}^d\alpha_j\,\frac{x_j}{|x|}\,\sqrt{m^2-\lambda^2}\,K_{\frac{d}2}\Big(\sqrt{m^2-\lambda^2}\;|x|\Big)\\
+(m\,\beta+\lambda\,\mathbb I_N)\,K_{\frac{d}2-1}\Big(\sqrt{m^2-\lambda^2}\;|x|\Big)\Bigg)
\end{multline*}
where $c_{d,\lambda}=\frac1{2\pi}\,\big(\frac{\sqrt{m^2-\lambda^2}}{2\pi}\big)^{d/2-1}$ and $K_\nu$ is the modified Bessel function of the second kind. In particular, there is $C\ge 0$ so that
\[
\big| R_0(\lambda)(x) \big|\le
\begin{cases}
C\,|x|^{1-d}&\mbox{as}\,|x|\to 0\,,\\
C\,\re^{-\sqrt{m^2-\lambda^2}\;|x|}&\mbox{as}\,|x|\to+\infty\,.
\end{cases}
\]
So, in the Dirac case, we have $R_0(\lambda)\in\mathrm L^r(\R^d)$ for all $r < \frac{d}{d-1}$ and $R_0(\lambda)\in\mathrm L^{\frac{d}{d-1}}_w(\R^d)$. In particular, the case~\ref{item:2-key_lemma} is not covered by~\ref{item:1-key_lemma} only in the case where $r=\frac{d}{d-1}$, which corresponds to the critical exponent case $p=d\ge 2$, that is, $q=\frac{2\,d}{d+2}$ in~\eqref{eq:Nbis}.
\end{remark}
\begin{remark} \label{rmk:J(1)>0} Let us consider the case $s=1$ in~\eqref{eq:J}.
In order to see that $J(1)>0$ in the Dirac case with $\lambda\in(-\,m,m)$, let $f\in\mathrm L^q(\R^d,\C)$ be a normalized function and let $\phi_+\in\C^N$ be a normalized vector such that $\beta\,\phi_+=\phi_+$. We find that
\[
(\Dirac+\lambda)\,f\,\phi_+=(m+\lambda)\,f \phi_++(-\,\ri\,\nabla f) \cdot \alpha\,\phi_+\,.
\]
Moreover, by~\eqref{eq:def:alphai_beta}, we have that $\big\langle \phi_+,\alpha_j\,\phi_+\big\rangle_{\C^N}=0$. Thus:
\be{eq:upper_component}
\begin{split}
J(1)\ge\big\langle f\,\phi_+, R_0(\lambda)\,f\,\phi_+\big\rangle&=
\big\langle f\,\phi_+, \big(-\Delta+m^2-\lambda^2\big)^{-1}\,(\Dirac+\lambda)\,f\,\phi_+\big\rangle\\
&=
(m+\lambda)\,\big\langle f, \big(-\Delta+m^2-\lambda^2\big)^{-1}\,f \big\rangle_{\mathrm L^2(\R^d,\C)} > 0\,.
\end{split}
\ee
\end{remark}

\begin{proof}[Proof of Lemma~\ref{lem:key_lemma}]
First, we note that the condition $K(x)=K(-x)^*$ reads $\widehat{K}(k)=\widehat{K}(k)^*$, so the operator $K$ is symmetric.

In the first part of the proof, we cover both cases~\ref{item:1-key_lemma} and~\ref{item:2-key_lemma} by assuming
\be{eq:weaker_condition}
K\in\mathrm L^r_w\big(\R^d, \cM_N(\C)\big) \cap \mathrm L^r\big(\cB_1^c, \cM_N(\C)\big)\quad \mbox{with}\quad \cB_R :=\{ x\in\R^d\,:\,|x|< R\}
\ee
 with $\frac2q+\frac1r=2$. From the Hardy-Littlewood-Sobolev inequality, and since $K\in\mathrm L^r_w(\R^d)$, we have
\be{eq:control_w1_w2}
\forall\,w_1\,,\;w_2\in\mathrm L^q(\R^d)\,,\quad\big|\big\langle w_1, K*w_2 \big\rangle \big|\le C\,\|w_1\|_q\,\|w_2\|_q\,\|K\|_{r,w}\,.
\ee
In particular, $w \mapsto\big\langle w, K*w\big\rangle$ is well-defined and real valued on $\mathrm L^q(\R^d)$.

Using the scaling $w_s=s^{1/q}\,w_1$, we obtain that
\be{eq:scaling-J} 
J(s)=s^{2/q}\,J(1)\,.
\ee
Since $J(1) > 0$, we deduce first that $J(s)$ is increasing. Also, since $2/q > 1$, $J(s)$ is convex and so we have the strong binding inequality
\be{eq:strong_binding}
\forall\,s\,,\; s' > 0\,,\quad J(s+s') > J(s)+J(s')\,.
\ee

Let $(w_n)_{n\in\N}$ be a maximizing sequence for $J(1)$. Our argument relies on the concentra\-tion-compactness method for the sequence $(w_n)_{n\in\N}$, following the approach of Lions~\cite{Lions-84} and using Levy's functional. It differs from the concentration-compactness method used in~\cite{ELS2021}, as we work directly with the Birman-Schwinger operator instead of the min-max quadratic form. We set
\[
Q(\rho) :=\liminf_{n \to+\infty} Q_n(\rho)\quad \mbox{with}\quad Q_n(\rho) :=\sup_{x\in\R^d}\int_{\cB(x,\rho)} | w_n |^q\,\rd x\,.
\]
It is clear from the definition that $\rho \mapsto Q(\rho)$ is non-decreasing, and that $Q(\rho)\le 1$ for all~$\rho>0$. We set
\[
\mu :=\lim_{\rho \to+\infty} Q(\rho)\in [0, 1]\,.
\]
We divide the proof in the classical steps of the concentration-compactness method and start by discarding the cases $\mu=0$ (vanishing) and $\mu<1$ (dichotomy).

\medskip\noindent\emph{\textbf{$\bullet$ Vanishing}}. 
Fix $\varepsilon :=J(1)/4 > 0$. Since $K\in\mathrm L^{r} (\cB_1^c)$, there is $R>1$ large enough so that
\[
\nrm{K}{\mathrm L^{r}(\cB_R^c)}\le \varepsilon\,.
\]
By Young's inequality, since $\frac2q+\frac1r=2$, we get that for all $w\in\mathrm L^q(\R^d)$ with \hbox{$\|w\|_q=1$},
\[
\big\langle w, (\Id_{\cB_R^c} K)* w \big\rangle\le \varepsilon\,\nrm wq^2\le \varepsilon\,.
\]
We now estimate the contribution of $\Id_{\cB_R}\,K$. For $z\in\Z^d$, let $C_z$ be the cube $z+[0, 1]^d$, so that $\{C_z \}_{z\in\Z^d}$ covers $\R^d$. For a function $w : \R^d \to \C^N$, we have
\begin{align*}
\big\langle w,\(\Id_{\cB_R}\,K\)*w\,\big\rangle
&=\sum_{z, z'\in\Z^d}\iint_{C_z\times C_{z'}}\big\langle w (x),(\Id_{\cB_R}\,K)(x-y)\,w(y)\big\rangle_{\C^N}\,\rd x\,\rd y\\
&\le\|K\|_{\mathrm L^r_w(\R^d)} \sum_{z, z'\in\Z^d} \nrm{w}{\mathrm L^q(C_z)}\,\nrm{w}{\mathrm L^q(C_{z'})}\,\Id_{\left\{|z-z'|\le R+2\,\sqrt{d}\right\}}
\end{align*}
using again the Hardy-Littlewood-Sobolev inequality. The double sum can be seen as a discrete convolution, and we apply Young's inequality with $z \mapsto\|w\|_{\mathrm L^q(C_z)}\in\ell^2(\Z^d)$ and $z \mapsto \Id_{\left\{| z |\le R+2\,\sqrt{d}\right\}}\in\ell^1(\Z^d)$ to bound
\[
\big\langle w, (\Id_{\cB_R}\,K)* w \big\rangle\le C_R \sum_{z\in\Z^d} \nrm{w}{\mathrm L^q(C_z)}^2
\le C_R \,\sup_{z\in\Z^d} \nrm{w}{\mathrm L^q(C_z)}^{2-q}\,\nrm wq^q
\]
where $C_R$ is a positive constant which is independent of $w$: for all $w\in\mathrm L^q(\R^d)$ with \hbox{$\|w\|_q=1$}, we have
\[
\big\langle w, K* w \big\rangle\le \varepsilon+C_R\,\sup_{z\in\Z^d} \nrm{w}{\mathrm L^q(C_z)}^{2-q}\,.
\]
Applying this estimate to a maximizing sequence $(w_n)_{n\in\N}$ for $J(1)=4\,\varepsilon$, we obtain that, up to a subsequence,
\[
\tfrac12\,J(1)\le\left\langle w_n, K* w_n \right\rangle\le \tfrac14\,J(1)+C_R \sup_{z\in\Z^d} \nrm{w_n}{\mathrm L^q(C_z)}^{2-q}\,.
\]
This implies
\[
Q_n\(\sqrt{d}\)\ge \sup_{z\in\Z^d} \nrm{w_n}{\mathrm L^q(C_z)}^q\ge \frac{J(1)}{\(4\,C_R)\)^{\frac q{2-q}}}>0
\]
and finally $\mu > 0$, which discards the \emph{vanishing} case of the concentration-compactness method.

\medskip\noindent\emph{\textbf{$\bullet$ Dichotomy}}. 
By definition of $Q_n$, there are sequences of centers $x_n\in\R^d$ and radii $\rho_n > 0$ going to infinity so that
\[
\lim_{n\to+\infty}\int_{\cB(x_n,\rho_n)} |w_n |^q\,\rd x=\mu\,.
\]
Without loss of generality, by translating the functions $w_n$, we may assume $x_n=0$. In addition, up to a non-displayed subsequence, we have that for all $\eps > 0$, there is $n_0$ large enough so that, for all $n\ge n_0$, we have
\[
\int_{\rho_n<|x|<2\,\rho_n}|w_n|^q\,\rd x<\eps\quad\mbox{and}
\quad\left|
1-\mu-\int_{|x|>2\,\rho_n}|w_n|^q\,\rd x\right|<\eps\,.
\]
We set
\[
\begin{cases}
w_n^{(1)}&:=w_n\,\Id_{\left\{|x|\le \rho_n\right\}}\,,\\
w_n^{(2)}&:=w_n\,\Id_{\left\{\rho_n\le |x|\le 2\,\rho_n\right\}}\,,\\
w_n^{(3)}&:=w_n\,\Id_{\left\{|x| > 2\,\rho_n\right\}}\,.
\end{cases}
\]
Introducing $E(w_1, w_2) :=\langle w_1, K* w_2 \rangle$ and $\cE(w):=E(w,w)$, we have
\[\label{eq:estimate-E(wn)}
\begin{split}
\cE(w_n)&=\cE\(w_n^{(1)}\)+\cE\(w_n^{(2)}\)+\cE\(w_n^{(3)}\)\\
&\quad+2\,{\rm Re} \(E\(w_n^{(1)}, w_n^{(2)}\)+E\(w_n^{(1)}, w_n^{(3)}\)+E\(w_n^{(2)}, w_n^{(3)}\)\)\\
&\le J(\mu)+J(\eps)+J(1-\mu+\eps)\\
&\quad+2\,{\rm Re} \(E\(w_n^{(1)}, w_n^{(2)}\)+E\(w_n^{(1)}, w_n^{(3)}\)+E\(w_n^{(2)}, w_n^{(3)}\)\)\,.
\end{split}
\]
{}From~\eqref{eq:control_w1_w2}, and the fact that $\|w_n^{(2)}\|_q\le \varepsilon^{1/q}$, we get that
\[
E\(w_n^{(1)}, w_n^{(2)}\)\le C\,\mu\,\varepsilon^{1/q}\quad\mbox{and}\quad E\(w_n^{(2)}, w_n^{(3)}\)\le C\, (1-\mu)\,\eps^{1/q}\,.
\]
Finally, we have
\begin{align*}
\left| E\(w_n^{(1)}, w_n^{(3)}\) \right|&\le\int_{|x|\le \rho_n}\int_{| y |\ge 2 \rho_n}\big|\big\langle w_n(x) , K(x-y)\,w_n(y) \big\rangle_{\C^N} \big|\,\rd x\,\rd y\\
&\le \iint_{\R^d \times \R^d} | w_n (x)|\,| w_n(y)|
\,(K\,\Id_{\cB_{\rho_n}^c})(x-y)\,\rd x\,\rd y\\
&\le\left\|K\,\Id_{\cB_{\rho_n}^c}\right\|_r\leq C\,\eps
\end{align*}
for $n$ large enough, where in the last line we used Young's inequality, and the fact that $\rho_n \to+\infty$.
Thanks to these facts, we can conclude that
\[
J(1)\le J(\mu)+J(1-\mu+\eps)+J(\eps)+C\,\eps^{1/q}\,.
\]
In the limit as $\eps \to 0$, we obtain $J(1)\le J(\mu)+J(1-\mu)$, which contradicts~\eqref{eq:strong_binding} if $\mu \neq 1$. So $\mu=1$, which discards the \emph{dichotomy} case of the concentration-compactness method.

\medskip\noindent\emph{\textbf{$\bullet$ Convergence for tight sequences}}.
At this point, we proved that for all $\eps > 0$ there is $\rho > 0$ and $n_0$ large enough so that, for all $n > n_0$, and after appropriate translations and subsequences,
\be{eq:Lq-norm-wn-small}
\nrm{\Id_{\cB_\rho^c}\,w_n}q\le \eps\,.
\ee
In other words, the sequence $(w_n)_{n\in\N}$ is tight in $\mathrm L^q(\R^d,\C^N)$. The sequence $(w_n)_{n\in\N}$ is bounded in the reflexive Banach space $\mathrm L^q(\R^d,\C^N)$. Hence, up to a non-displayed subsequence, $(w_n)_{n\in\N}$ converges weakly to some $w\in\mathrm L^q(\R^d,\C^N)$, and we have $\|w\|_q\le 1$.

Let us prove that $\cE(w)=J(1)$. Let $\eps > 0$, and let $\rho > 0$ be large enough so that~\eqref{eq:Lq-norm-wn-small} holds. In particular, by Hardy-Littlewood-Sobolev, we have
\[
\left| \langle w_n\,\1_{\cB_\rho^c}, K*w_n \big\rangle \right|\le C\,\|w_n\|_q\,\|w_n\,\1_{\cB_\rho^c}\|_q\le C\,\varepsilon\,,
\]
and we have a similar inequality with $w$ instead of $w_n$. On the other hand, we have
\[
\langle w_n\,\1_{\cB_\rho}, K*w_n \big\rangle=\langle w_n, T w_n \big\rangle_{\mathrm L^q, \mathrm L^{q'}}\,,
\]
where $T$ is the operator from $\mathrm L^q(\R^d)$ to $\mathrm L^{q'}(\R^d)$ with kernel $T(x,y)=\1_{\cB_\rho}(x)\,K(x-y)$. The operator $T:\mathrm L^q(\R^d,\C^N)\to\mathrm L^{q'}(\R^d,\C^N)$ is bounded. We claim that $T$ is a compact operator. In the Dirac case~\ref{item:2-key_lemma}, this comes from the fact that $K*w_n\in\mathrm W^{1, q}$ with $\|K*w_n\|_{\mathrm W^{1, q}}(\R^d,\C^N)\le C\,\|w_n\|_q$ together with the Rellich-Kondrachov compact embedding theorem. In the case~\ref{item:1-key_lemma}, where $K\in\mathrm L^r(\R^d)$, setting $\tau_h\,f(x) :=f(x-h)$, we have
\[
\big\|\tau_h(K*w)-K*w\big\|_{\mathrm L^{q'}(B_\rho)}=\big\|\(\tau_h\,K-K\)*w\big\|_{q'}\le\|\tau_h\,K-K\|_r\,\|w\|_q\,.
\]
Since $K\in\mathrm L^r(\R^d)$, we have $\|\tau_h\,K-K\|_r \to 0$ as $h \to 0$, and we conclude with the Kolmogorov-Riesz-Fréchet theorem (see for instance~\cite[Theorem~4.26]{Brezis_2011}).

As a consequence, $(T w_n)_{n\in\N}$ converges strongly to $T w$ in $\mathrm L^{q'}(\R^d,\C^N)$. In particular, we obtain that
\[
\lim_{n \to+\infty}\big\langle w_n\,\1_{\cB_\rho}, K*w_n \big\rangle=\big\langle w\,\1_{\cB_\rho}, K*w \big\rangle.
\]
Gathering the two inequalities gives
\begin{align*}
&\left| \big\langle w_n, K* w_n \big\rangle-\big\langle w, K* w \big\rangle \right|\\
&\;\le\left| \big\langle w_n\,\1_{\cB_\rho}, K* w_n \big\rangle-\big\langle w\1_{\cB_\rho} , K* w \big\rangle \right|+\left| \big\langle w_n\,\1_{\cB_\rho}^c, K* w_n \big\rangle \right|+\left| \big\langle w\,\1_{\cB_\rho}^c, K* w \big\rangle \right|\\
&\;\le\left| \big\langle w_n\,\1_{\cB_\rho}, K* w_n \big\rangle-\big\langle w\1_{\cB_\rho} , K* w \big\rangle \right|+2\,C\,\varepsilon.
\end{align*}
Sending first $n$ to $+\infty$, and then $\varepsilon$ to $0$ shows that $\langle w, K* w \rangle=J(1)$. Finally, since $\|w\|_q\le 1$, by~\eqref{eq:scaling-J} we deduce that $\|w\|_q=1$. This proves that $(w_n)_{n\in\N}$ converges strongly to $w$ in $\mathrm L^q(\R^d,\C^N)$ and that $w$ is an optimizer.
\end{proof}

It is an open question to decide whether $T$ is compact or not under the condition~\eqref{eq:weaker_condition}.

\begin{proof}[Proof of Theorem~\ref{th:w}.]
In the setting of Theorem~\ref{th:w}, we take $K = R_0(\lambda)$. We have $J(1) > 0$ in this case, as noticed in Remark~\ref{rmk:J(1)>0}, so by Lemma~\ref{lem:key_lemma}, the problem~\eqref{eq:Nbis} admits maximizers. By standard arguments, optimizers satisfy the Euler-Lagrange equation~\eqref{eq:EL_for_w}.

The fact that $\lambda \mapsto \cN(\lambda,p) $ is strictly increasing comes from the fact that $\lambda \mapsto R_0(\lambda)$ is operator strictly increasing: for instance, we have $\partial_\lambda R_0(\lambda)=(R_0(\lambda))^2 > 0$. Let us prove the continuity. Let $-\,m < \lambda' < \lambda < m$, and let $w_\lambda$ be the optimizer for $\cN(\lambda,p)$. Using that $\cN(\cdot)$ is strictly increasing and the resolvent identity
$$R_0(\lambda') = R_0(\lambda) - (\lambda - \lambda')\,R_0(\lambda')\,R_0(\lambda)\,,$$
we obtain
\[
 0 < \cN(\lambda, p) - \cN(\lambda', p) \le (\lambda - \lambda') \,\big\langle w_\lambda, R_0(\lambda')\,R_0(\lambda)\,w_\lambda \big\rangle\,.
\]
Using that $R_0$ is a bounded operator from $\mathrm L^q(\R^d)$ to $\mathrm L^{q'}(\R^d)$, and from $\mathrm L^{q'}(\R^d)$ into itself, with uniform bounds in a neighborhood of $\lambda$, we deduce that there is $C > 0$ so that
\[
 \left| \big\langle w_\lambda, R_0(\lambda')\,R_0(\lambda)\,w_\lambda \big\rangle \right| \le C\,\| w_\lambda \|^2_q = C\,,
\]
This proves that $\cN(\cdot,p)$ is locally Lipschitz, hence continuous.

We now prove the bounds on $\lim_{\lambda \to \pm m} \cN(\lambda, p)$. To prove that $\lim_{\lambda \to m} \cN(\lambda, p) = + \infty$, we go back to~\eqref{eq:upper_component} and take a function $f=L^{-d/q}\,g(\cdot/L)$, where $g$ is an arbitrary test function that is normalized in $\mathrm L^q(\R^d)$. This gives
$$
\mathcal{N}(\lambda, p) \ge L^{-2\,d/q}\,(m+\lambda )\,\Big\langle{g(\cdot/L)},{\big(-\Delta +m^2 -\lambda^2\big)^{-1}\,g(\cdot/L)}\Big\rangle\,.
$$
We bound the resolvent as
$$
\big(-\Delta +m^2 -\lambda^2\big)^{-1} \ge \big(m^2-\lambda^2\big)^{-1}\left(1 + \big(m^2-\lambda^2\big)^{-1}\Delta\right)
$$
and change variables to obtain
$$
\mathcal{N}(\lambda, p) \ge L^{d\,(1 - \frac{2}{q})}\,(m-\lambda )^{-1}\left(\nrm{g}2^2 -L^{-2}\,\big(m^2-\lambda^2\big)^{-1}\,\nrm{\nabla g}2^2\right).
$$
Since $1-\frac2q=p$, we may take $L = (m-\lambda)^{-\alpha}$ for any $\alpha \in (1/2, p/d)$ and conclude that
$\lim_{\lambda \to m} \mathcal{N}(\lambda, p) = +\infty$.

Finally, to prove that $\lim_{\lambda \to -m} \cN(\lambda, p) > 0$, we claim that
\be{Existence:w}\mbox{\emph{there exists a function $w \in\mathrm L^2\cap\mathrm L^q(\R^d, \C^N)$ such that $\| w \|_q = 1$ and $P\,w = w$,}}
\ee
where $P := \1_{m < \Dirac < 2\,m}$ is the spectral projection of the free Dirac operator onto $(m, 2\,m)$. This would give
\[
 \cN(\lambda, p) \ge \langle w, R_0(\lambda)\,w \rangle = \left\langle w, \dfrac{P}{\Dirac - \lambda}\,w \right\rangle + \left\langle w, \dfrac{P^\perp}{\Dirac - \lambda}\,w \right\rangle.
\]
The second term is null since $P^\perp\,w = 0$. For the first term, we have $m < \Dirac < 2\,m$ on the range of $P$, and in particular $P\,(\Dirac - \lambda)^{-1}\,P \ge P\,(2\,m - \lambda)^{-1}\,P$, hence $\cN(\lambda, p) \ge (2\,m - \lambda)^{-1}\,\| w \|_2$. Taking $\lambda \to -m$ shows that $\lim_{\lambda \to -m} \cN(\lambda, p) \ge (3\,m)^{-1}\,\| w \|_2> 0$.

It remains to prove~\eqref{Existence:w}. Recall that $\Dirac = \cF M(k)\,\cF^*$, where $\cF$ denotes the Fourier transform and $M(k)$ is the $d \times d$ matrix $M(k) := \boldsymbol\alpha \cdot k + m\,\beta$, which satisfies $M(k) = M(k)^*$, $M(k)^2 = (| k |^2 + m^2)\,\bbI_d$, and $\sigma(M(k)) = \big\{ \pm\,(| k |^2 + m^2)^{1/2} \big\}$. 
Let \hbox{$v \mapsto v(k)$} be a smooth family of spinors from some open ball $\cB(k = 0, \varepsilon)$ to $\C^d$, with $0 < \varepsilon < m$, so that $M(k)\,v(k) = (| k |^2 + m^2)^{1/2}\,v(k)$. To construct such a local family of spinors, one can consider $v_0$ a normalized eigenfunction of $M(k=0)$, and set,
\[
    v(k) := \frac{P(k)\,v_0}{\| P(k)\,v_0 \|_2}\,, \quad P(k) := \1 (M(k) > 0)\,.
\]
Since $k \mapsto P(k)$ is smooth locally around $0$ ($P(k)$ can be written as a Cauchy integral $P(k) = (2\,\ri\,\pi)^{-1} \oint_{\sC} \big(z - M(k)\big)^{-1}\,\rd z$ with a contour enclosing $m$), so is $k \mapsto v(k)$.
Let also $\chi(k) : \R^d \to \R^+$ be a non null smooth compactly supported function, with $\chi(k) = 0$ for $| k | > \varepsilon$. We consider the function
\[
 w := \dfrac{\widetilde{w}}{\| \widetilde{w} \|_q} \quad \text{with} \quad \widetilde{w} := \cF \big( \chi(k)\,v(k) \big)\,.
\]
By construction, we have $\widetilde{w} \neq 0$, and since $\widetilde{w}$ has a Fourier transform which is smooth and compactly supported, it belongs to the Schwartz class $\cS(\R^d, \C^N)$. Finally, since on the support of $\chi$, we have $M(k)\,v(k) = (| k |^2 + m^2)^{1/2}\,v(k)$ with $m < (| k |^2 + m^2)^{1/2} < \sqrt{2}\,m^2$, we deduce that
\[
 P\,\widetilde{w} = \cF \big( \1_{m < M(k) < 2\,m}\,\chi(k)\,v(k) \big) = \cF \,\big(\chi(k)\,v(k) \big) = \widetilde{w}\,,
\]
which concludes the proof of~\eqref{Existence:w}.
\end{proof}

\subsection{Regularity of the solutions of the non-linear Dirac equation}\label{Sec:Regularity}

Under Condition~\eqref{eq:cond_p_ess_sa}, solutions of~\eqref{eq:def:KellerDirac} with $\Psi\in\mathrm L^2(\R^d,\C^N)$ are in $\Dom (\Dirac-V)=\mathrm H^1(\R^d,\C^N)$. Let us consider the other cases of Proposition~\ref{prop:selfadjoint}. If $d=1$ and $1<p\le2$, any optimal function for~\eqref{eq:Nbis} obtained in Theorem~\ref{th:w} gives rise to a solution $\Psi\in\mathrm W^{1,q}(\R,\C^2)$ of~\eqref{eq:def:KellerDirac} with $q=2\,p/(p+1)$. We conclude that $\Psi$ is continuous. If $p=d=2$ and $q=4/3$, the corresponding solution $\Psi$ of~\eqref{eq:def:KellerDirac} is in $\mathrm W^{1,q}(\R,\C^2)\hookrightarrow\mathrm H^{1/2}(\R,\C^2)$, hence $V\,|\Psi|^2=|\Psi|^{2\,p/(p-1)}$ is integrable and $\Psi\in\Dom(\Dirac-V)$ of the distinguished extension of Proposition~\ref{prop:selfadjoint} but we do not know whether $\Psi \in \mathrm H^1(\R^2,\C^2)$ or not. 

In dimension $d=1$, an explicit expression of the solutions of~\eqref{eq:def:KellerDirac} such that 
\[
\lim_{x\to\pm\infty}\Psi(x)=(0,0)^\top
\]
is given in Theorem~\ref{th:alphac_d=1}. In the case $p=d=2$, it is unclear how to obtain $\Psi\in\mathrm H^1(\R^d,\C^N)$ by general arguments, as pointed out in~\cite{MR3794033}. However, any solution to~\eqref{eq:def:KellerDirac} (and not only the ones found in Theorem~\ref{th:w}) have additional regularity properties under Condition~\eqref{eq:cond_p_ess_sa}.
\begin{proposition} \label{thm:regularity}
Let $\lambda\in [-\,m, m)$ and either $p\ge d$ if $d\ge3$, or $p > d$ in dimension $d=1$ and $2$. If $\Psi\in\mathrm H^1(\R^d,\C^N)$ solves~\eqref{eq:def:KellerDirac}, then $\Psi\in C^\infty(\R^d,\C^N)$.
\end{proposition}
\begin{proof}
Let us first prove that $\Psi\in\mathrm L^\infty(\R^d,\C^N)$ with a usual bootstrap argument. If $\Psi\in\mathrm L^q(\R^d,\C^N)$, then $| \Psi |^\frac2{p-1}\,\Psi\in\mathrm L^{\frac{p-1}{p+1}\,q}(\R^d,\C^N)$. Also, if $q > 2\,\frac{p+1}{p-1}$, then $2 < \frac{p-1}{p+1}\,q < q$, so if $\Psi\in\mathrm L^2(\R^d,\C^N) \cap \mathrm L^q(\R^d,\C^N)$, then $\lambda\,\Psi+| \Psi |^\frac2{p-1}\,\Psi\in\mathrm L^{\frac{p-1}{p+1}\,q}(\R^d,\C^N)$. In particular, $\Dirac \Psi\in\mathrm L^{\frac{p-1}{p+1}\,q}(\R^d,\C^N)$, hence $\Psi\in\mathrm W^{1, \frac{p-1}{p+1}\,q}(\R^d,\C^N) \hookrightarrow \mathrm L^{\tilde q}(\R^d,\C^N)$, with $\tilde q=+\infty$ if $\frac{p-1}{p+1}\,q > d$ and
\[
\frac1{\tilde q}=\frac{p+1}{p-1}\,\frac1q-\frac1d
\]
otherwise. As a first step of an iteration scheme, we proved that if $\Psi\in\mathrm L^q(\R^d,\C^N)$, then $\Psi\in\mathrm L^{\tilde q}(\R^d,\C^N)$ as well. For the initialization, we note that $\mathrm H^1(\R^d,\C^N) \hookrightarrow \mathrm L^q(\R^d,\C^N)$ for all $q$ such that $2\le q\le \frac{2\,d}{d-2}=:2^*$ if $d\ge3$ and $2\le q <+\infty=:2^*$ if $d=2$. Hence with $2\,\frac{p+1}{p-1}<2^*$, there is $q_0 > 2\,\frac{p+1}{p-1}$ so that $\Psi\in\mathrm L^2(\R^d,\C^N)\cap \mathrm L^{q_0}(\R^d,\C^N)$. The map $F : x \mapsto \frac{p+1}{p-1}\,x-\frac1d$ satisfies $F(x) < x$ for $x\in [0, x^*]$ with $x^*=\frac{p-1}{2\,d} < \frac12\,\frac{p-1}{p+1}$. We easily deduce that there is $n\in\N$ so that $F^{(n)}(\frac1{q_0}) < 0$, which proves $\Psi\in\mathrm L^\infty(\R^d,\C^N)$ as wanted.

Since $\Dirac \Psi\in\mathrm L^\infty(\R^d,\C^N)$, we have $\Psi\in\mathrm W^{1,\infty}(\R^d,\C^N) \hookrightarrow C^{0, \alpha}(\R^d,\C^N)$ for all $0\le \alpha < 1$, by bootstrapping again, we obtain $\Psi\in C^\infty(\R^d,\C^N)$.
\end{proof}

\section{Lieb-Thirring inequality}\label{Sec:LT}

This section contains the proof of Theorem~\ref{thm:LT-bound}.
We closely follow the original proof by Lieb and Thirring~\cite{LieThi-75,Lieb-Thirring76} (see also~\cite{LieSei-10}). This is possible since we are assuming $V \ge 0$. In the general case where $V$ has no sign, some results can be found in the works of Cuenin~\cite{cuenin2017eigenvalue}, and Frank-Simon~\cite{MR2785827}, where the authors control the Riesz-mean
$$
\sum_k {\rm dist}\big( \lambda_k, \sigma(\Dirac - V)\big)^\gamma\,,
$$
that is, the distance to the whole spectrum. Actually, without assuming a sign on $V$, one cannot expect to control the sums in~\eqref{eq:LT_RieszMean}, since, for $V\le 0$ small, the eigenvalues of \hbox{$\Dirac - V$} emerge from the bottom essential spectrum (hence have a distance of order \hbox{$2\,m>0$} to the upper essential spectrum).
Here, since $V$ is nonnegative, the eigenvalues emerge from the upper essential spectrum as the strength of the potential increases. 

\begin{proof}[Proof of Theorem~\ref{thm:LT-bound}]
It is sufficient to prove the result for $V$ bounded and compactly supported. By the \emph{Birman-Schwinger principle} introduced in Section~\ref{Sec:Birman-SchwingerOperator}, we know that $\lambda$ is an eigenvalue for $\Dirac - V$ acting on $\mathbb C^N$ valued spinors if and only if $1$ is an eigenvalue of $K_V(\lambda)$ defined by~\eqref{eq:def:KV}: see Proposition~\ref{Prop:Birman-Schwinger}. We also proved that $\lambda \mapsto K_V(\lambda)$ is operator increasing. In particular, if we set
\[
N_e(V) := \text{number of eigenvalues of $\Dirac - V$ in $[-\,m, m - e]$}
\]
and
\[
B_e(V) := \text{number of eigenvalues of $K_V(m - e)$ which are greater or equal than $1$}\,,
\]
then we have $N_e(V) \le B_e(V)$. We have equality if the highest eigenvalues of $K_V(\lambda)$ gets strictly smaller than $1$ as $\lambda \to -\,m$. This happens for instance if $\| V \|_p \le \alpha_*(p)$. 

With $R_0$ defined by~\eqref{eq:def:Rz}, using the operator inequality
$$
R_0(\lambda) \le \mathbbm{1}_{\C^N}\(\sqrt{-\Delta + m^2}-\lambda\)^{-1}\,,
$$
we can estimate $B_e(V)$ by $N\,B_e^{\rm pr}(V)$, where $ B_e^{\rm pr}(V)$ is the number of eigenvalues above~$1$ of the pseudo-relativistic Birman-Schwinger operator 
$$
 K^{\rm pr}_V (m-e):= \sqrt{V} \(\sqrt{-\Delta + m^2}-m + e\)^{-1} \sqrt{V}\,.
$$
In addition, with the definition
\[
 N_e^{\rm pr}(V) := \text{number of eigenvalues of $\left( \sqrt{-\Delta + m^2} - m \right) - V$ less or equal than $-\,e$}\,,
\]
the usual Birman-Schwinger principle shows that $B_e^{\rm pr}(V) = N_e^{\rm pr}(V)$. To sum up, we have
\begin{equation} \label{eq:Ne_vs_Nepr}
 N_e(V) \le B_e(V) \le N \, B_e^{\rm pr}(V) = N \, N_e^{\rm pr}(V)\,.
\end{equation}
The operator $\sqrt{-\Delta + m^2} - m$ is sometimes called the Chandrasekhar (or pseudo-relativis\-tic) kinetic energy operator. It is a positive operator, $\sqrt{-\Delta + m^2} - m - V$ is bounded from below, and the min-max formula applies. We can now repeat the usual arguments of Lieb and Thirring for the pseudo-relativistic operator.

First, for $\gamma > 0$, the cake-layer representation gives
\be{eq:layer-cake}
 \sum_{k\ge1} e_k^\gamma = \gamma \int_{0}^{2\,m} e^{\gamma-1}\,N_e(V)\,\rd e \le \gamma\,N \, \int_{0}^{2\,m} e^{\gamma-1}\,B_e^{\rm pr}(V)\,\rd e\,.
\ee
Note that for the pseudo-relativistic model, if $-\,e_1^{\rm pr} \le -\,e_2^{\rm pr} \le \cdots < 0$ are the negative eigenvalues of $\left( \sqrt{-\Delta + m^2} - m \right) - V$, we have
\[
 \sum_{k\ge1} (e_k^{\rm pr})^\gamma = \gamma \int_{0}^{\infty} e^{\gamma-1}\,N_e^{\rm pr}(V)\,\rd e = \gamma \, \int_{0}^{\infty} e^{\gamma-1}\,B_e^{\rm pr}(V)\,\rd e\,,
\]
and the integral runs over $e \in \R^+$ instead of $e \in (0, 2\,m)$. Actually, the previous two inequalities together with~\eqref{eq:Ne_vs_Nepr} show that
\[
 \sum_{k\ge1} e_k^\gamma \le N \sum_{k\ge1} (e_k^{\rm pr})^\gamma\,.
\]
In other words, the Riesz-mean of the eigenvalues increases when one replaces the Dirac operator by the pseudo--relatisvistic one (up to the $N$ factor). Lieb-Thirring inequalities for the last sum have been derived by Daubechies in~\cite{Daubechies_83} (and used, \emph{e.g.}, in~\cite{Lieb2001}). In what follows, we derive another inequality specifically for the Dirac operator. We use in particular the fact that the integral in~\eqref{eq:layer-cake} only runs for $e$ in the bounded interval $(0, 2\,m)$ instead \hbox{of~$\R^+$}.

\medskip\noindent$\bullet$ \emph{Bound for $B_e^{\rm pr}(V)$}. Assume $V \in\mathrm L^p(\R^d)$ with $d < p$. The number of eigenvalues above $1$ of $K_V^{\rm pr}(m - e)$ is bounded from above by $\| K_V^{\rm pr}(m-e) \|_{\fS^p}^p$. We estimate this norm using the Kato-Simon-Seiler inequality (see~\cite[Theorem 4.2]{SimonTraceIdeals}). Using a decomposition similar to the one in the proof of Lemma~\ref{lem:KV}, we obtain
\[
 B_e^{\rm pr}(V) \le \| K_V^{\rm pr}(m - e)\,\|_{\fS_{p}}^p \le C_p \left\| g_{m, e} \right\|_{p}^p\,\| V \|_p^p\,,
\]
where we introduced the function
\[
 g_{m, e}(k) := \left( \sqrt{k^2 + m^2} - (m - e) \right)^{-1}\,.
\]
Note that $g_{m,e} \in \mathrm L^p(\R^d)$ since $p>d$, and
\[
 \| g_{m,e} \|_{p}^p = \int_{\R^d} \dfrac{\rd k}{\left( \sqrt{k^2 + m^2} - m + e \right)^p}
 = | \SS^{d-1} | \int_0^\infty \dfrac{r^{d-1}\,\rd r}{\left( \sqrt{r^2 + m^2} - m + e \right)^p}\,.
\]
To estimate this norm, we make the change of variable $X = \frac{1}{e}\,\big( \sqrt{r^2 + m^2} - m\big)$, so that $r = \sqrt{(e\,X + m)^2 - m^2} = \sqrt{e\,X\,(e\,X + 2\,m)}$. We obtain
\[
 \| g_{m,e} \|_{p}^p = \frac{| \SS^{d-1} |}{e^{p - \frac{d}{2}}} \int_0^\infty \dfrac{ \big( X (e\,X + 2\,m) \big)^{\frac{d}{2} - 1} (e\,X + m)\,\rd X}{\left( X+1 \right)^p}\,.
\]
The last integral is an increasing function of $e$ (and has a finite value as $e \to 0$ by the monotone convergence theorem). Since $e \in (0, 2\,m)$, we can bound this integral by its value at $e = 2\,m$. We deduce that there is a constant $C_{p,d}$ such that 
\be{eq:Ne-bound}
 B_e^{\rm pr} (V) \le C_{p,d}\,\nrm{V}p^p\,\dfrac{m ^{d/2}}{e^{p-d/2}}\,.
\ee

\medskip\noindent$\bullet$ \emph{Proof of the Lieb--Thirring estimate}. We now follow \cite{LieThi-75, Lieb-Thirring76, LieSei-10}. The min--max principle for the pseudo-relativistic operator shows that its eigenvalues are decreasing when $V$ increases. Since $V \le [V-e/2]_+ +e/2$, we may bound
\begin{multline*}
 B_e^{\rm pr}(V) = N_{e}^{\rm pr} (V) \le N_{e}^{\rm pr}\big([V-e/2]_+ + e/2\big)\\ 
 = N_{e/2}^{\rm pr}\big([V-e/2]_+\big)
 = B_{e/2}^{\rm pr}\big([V - e/2]_+\big)\,.
\end{multline*} 
For any $p > d$, we can apply the bound in~\eqref{eq:Ne-bound} to estimate $B_{e/2}^{\rm pr}\big([V-e_2]_+\big)$. Inserting this estimate into~\eqref{eq:layer-cake}, we get
\begin{multline*}
 \sum_{k\ge1} e_k^\gamma 
 \le N\,C_{p,d}\,\gamma\, m^{\frac d2} \int_{0}^{2\,m }(e/2)^{\gamma-1+\frac d2 - p}\,\big\|[V-e/2]_+\big\|_p^p\,\rd e \\
 = C_{\gamma, d,p} \, m^{\frac d2} \int_{\R^d} \int_{0}^{2\,m }e^{\gamma-1+\frac d2 - p}\,\big[V(x)-e/2\big]_+^p\,\rd e \,\rd x \\
 \le C_{\gamma, d,p} \, m^{\frac d2} \int_{\R^d} V^{\gamma + \frac d2} (x) \int_0^{s^*(x)} s^{\gamma -1 + \frac d2 -p}\,(1-s)^p_+\,\rd s \,\rd x\,,
 \end{multline*}
where $s^*(x) := \min\{m/V(x), 1\}$, with the convention that $s^*(x)= 1$ if $V(x)= 0$. The second integral converges whenever $p < \gamma + d/2$. 
We can simply use the bound $(1-s)^p \le 1$ in the last integral, and finally obtain
\begin{equation} \label{eq:bound_LT}
\sum_{k\ge1} e_k^\gamma \le L_{\gamma, d,p}\,m^\frac d2\int_{\R^d} V_m^{\gamma + \frac d2 - p}\,V^{p}\,\rd x\quad \text{with} \quad V_m := \min \left\{ m, V \right\}\,.
\end{equation}
This inequality is valid for all $d < p < \gamma + d/2$. Note that $C_{\gamma, d, p}$ stays bounded in the limit as $p \to \gamma + d/2$, so a similar inequality also holds if $p = \gamma + d/2$. \end{proof}
\begin{remark} The result of Theorem~\ref{thm:LT-bound} can be extended to the case of a potential $V \in \mathrm L^p(\R^d, \R^+) + \mathrm L^{\gamma + d/2}(\R^d, \R^+)$ by noticing that the right-hand side of~\eqref{eq:bound_LT} is continuous for $V$ in this space.
\end{remark}

\section{Explicit computations}\label{Sec:Explicit}

\subsection{The case \texorpdfstring{$d=1$}{d=1}: proof of \texorpdfstring{Theorem~\ref{th:alphac_d=1}}{Theorem~1.2}}\label{Sec:Explicit1}

In this section, we prove the uniqueness and the symmetry up to translations of the solution of the nonlinear Dirac equation~\eqref{eq:def:KellerDirac}. We also compute the map $\alpha_D(\lambda,p)$.

\begin{proof}[Proof of Theorem~\ref{th:alphac_d=1}] In the one-dimensional case, Equation~\eqref{eq:def:KellerDirac} can be rewritten for the components of $\Psi=: (\varphi, \chi)^\top$ as
\be{eq:system-1-d}
\begin{cases}
\varphi'=-\,\Big(\lambda+m+\(|\chi|^2+|\varphi|^2\)^\frac1{p-1}\Big)\,\chi\,,\\
\chi'=\Big(\lambda-m+\(|\chi|^2+|\varphi|^2\)^\frac1{p-1}\Big)\,\varphi\,.
\end{cases}
\ee
The corresponding potential is $V=\(|\chi|^2+|\varphi|^2\)^\frac1{p-1}$. This system conserves
\[
\begin{split}
H(\varphi,\chi)&:=m\(|\chi|^2-|\varphi|^2\)+\lambda \(|\chi|^2+|\varphi|^2\)+\tfrac{p-1}p\,\(|\chi|^2+|\varphi|^2\)^\frac p{p-1}\,,\\
G(\varphi,\chi)&:=\bar \chi\,\varphi-\bar \varphi\,\chi\,.
\end{split}
\]
Since we are looking for solutions vanishing at $\pm\infty$, they satisfy $H\big(\varphi(x),\chi(x)\big)=0$, $ G(\varphi(x), \chi(x))=0$ for all $x\in\R$. This second condition shows that solutions can be chosen real valued. For real valued variables in the $(\varphi,\chi)$-plane, the level set $H(\varphi,\chi)=0$ has the shape of an infinity sign. Among real valued functions, uniqueness up to translations follows from the phase plane analysis. We can choose the unique solution with $\chi(0)=0$, $\varphi(0) > 0$, given. For this solution, $\varphi$ is even and $\chi$ is odd and positive on $\R^+$. Hence symmetry and uniqueness, up to translations and multiplication by a phase, are granted by elementary considerations. Next, we have
\[
V'=\tfrac1{p-1}\,\frac{(\chi^2+\varphi^2)'}{( \chi^2+\varphi^2)^\frac p{p-1}}\quad \mbox{with}\quad \(\chi^2+\varphi^2\)'=-\,4\,m\,\chi\,\varphi\,,
\]
which proves that $V$ is increasing in the quadrant $\{\chi<0,\varphi>0\}$ and decreasing in the quadrant $\{\chi>0,\varphi>0\}$. Hence $V$ is even and decreasing on $\R^+$, while on $\R^+$ both $\chi$ and $\varphi$ are positive valued.

\medskip Now let us compute $\nrm Vp$. It is enough to do the computation on $\R^+$. First, the equation $H(\varphi,\chi)=0$ can be rewritten as
\[
2\,m\,\varphi^2=(m+\lambda)\,V^{p-1}+\tfrac{p-1}p\,V^p\,,
\]
and so
\[
\varphi=\sqrt{\tfrac1{2\,m}\,V^{p-1}\(m+\lambda+\tfrac{p-1}p\,V\)}\,.
\]
Next, from the equation $V^{p-1}=\chi^2+\varphi^2$, we deduce that
\[
\chi=\sqrt{V^{p-1}-\varphi^2}=\sqrt{\tfrac1{2\,m}\,V^{p-1}\(m-\lambda-\tfrac{p-1}p\,V\)}\,.
\]
Finally, we have
\[
(p-1)\,V^{p-2}\,V'=\(V^{p-1}\)'=\(\chi^2+\varphi^2\)'=-\,4\,m\,\chi\,\varphi\,.
\]
Collecting the three last equalities shows that $V$ solves the autonomous differential equation
\[\label{eq:ode_V}
V'=-\,\tfrac2{p-1}\,V\,\sqrt{\(m-\lambda-\tfrac{p-1}p\,V\)\(m+\lambda+\tfrac{p-1}p\,V\)}\,.
\]
At $x=0$, we have $V'(0)=0$, which implies
\[
V(0)=\tfrac p{p-1}\,(m-\lambda)\,.
\]

\medskip\noindent\emph{\textbf{$\bullet$ Subcritical regime $\lambda >-\,m$}}.
The function
\[
Z(x) :=\tfrac{p-1}{p\,(m+\lambda)}\,V \big( \tfrac{p-1}{2\,(m+\lambda)}\,x \big)
\]
satisfies
\be{eq:ode_z}
Z'=-\,Z\,\sqrt{(z_0-Z)\,(1+Z)}\,,\quad Z(0)=z_0=\tfrac{m-\lambda}{m+\lambda}\,.
\ee
One can directly check that the solution of~\eqref{eq:ode_z} is
\[
Z(x)=\dfrac{2\,z_0}{(1+z_0) \cosh\(\sqrt{z_0}\,x\)+1-z_0}\,.
\]
This gives~\eqref{eq:explicit_V_noncritical}. The $\mathrm L^p(\R)$ norm of $V$ is computed as
\[
\nrm Vp^p=\frac{p^p\,(m+\lambda)^{p-1}}{2\,(p-1)^{p-1}}\,\nrm Zp^p\,.
\]
Using that $Z$ is even, monotone decreasing on $\R^+$, with the change of variable $z=Z(x)$ and $t=z/z_0$, we obtain, using~\eqref{eq:ode_z},
\begin{align*}
\|Z\|_{p}^p&=2\int_0^{+\infty} Z^p(x)\,\rd x\\
&=2\int_0^{z_0} \dfrac{z^{p-1}}{\sqrt{(z_0-z)\,(1+z)}}\,\rd z
=2\,z_0^{p-\frac12}\int_0^1 \frac{t^{p-1}}{\sqrt{(1-t)\,\big(1-(-z_0)\,t\big)}}\,\rd t\\
&=2\,z_0^{p-\frac12}\,B\(\tfrac12,p\)\,_2F_1\(\tfrac12,p;p+\tfrac12;-z_0\)\,.
\end{align*}
See~\cite[15.3.1 p.~558]{zbMATH03863589} for the last equality. This completes the computation of $\alpha_D(\lambda,p)$. By taking the limit as $p\to1_+$, we obtain $\alpha_D(\lambda,1)=\arccos(\lambda/m)$.

\medskip\noindent\emph{\textbf{$\bullet$ Critical case $\lambda=-\,m$}}. The function
\[
Z(x) :=\tfrac{p-1}{2\,m\,p}\,V\big(\tfrac{p-1}{2\,m}\,x\big)
\]
solves
\[
Z'=-\,2\,Z^{3/2}\,\sqrt{1-Z}\,,\quad Z(0)=1
\]
on $\R^+$. The solution is
\[
\forall\,x\in\R\,,\quad Z(x)=\dfrac1{1+x^2}\,.
\]
This gives~\eqref{eq:explicit_V_critical}, and the expression of $\alpha_\star(p)$ follows from
\[
\nrm Vp^p=\frac{p^p\,(2\,m)^{p-1}}{(p-1)^{p-1}}\,\nrm Zp^p
\]
with
\[
\|Z\|_p^p=\int_{\R}\dfrac{\rd x}{\(1+x^2\)^p}=B\(\tfrac12,p-\tfrac12\)
\]
according to~\cite[8.380.3 p.~917]{MR3307944}. This concludes the proof of Theorem~\ref{th:alphac_d=1}.
\end{proof}
Notice that $\lim_{z_0\to+\infty}\sqrt{z_0}\,B\(\tfrac12,p\)\,_2F_1\(\tfrac12,p;p+\tfrac12;-z_0\)=B\(\tfrac12,p-\tfrac12\)$, so that $\lim_{\lambda\to(-1)_+}\alpha_D(\lambda,p)=\alpha_\star(p)$.

\subsection{The radial case in dimension \texorpdfstring{$d=2$}2}\label{Sec:Explicit2}

We now provide some numerical simulations to obtain upper bounds for the maps $\Lambda_D(\alpha, p)$.

First, we restrict the minimization problem~\eqref{eq:min_vap_Dirac} to radial potentials, that is, we compute
\[
\Lambda_D^{\rm rad}(\alpha,p):=\inf\Big\{\lambda_D(V)\,:\,V\in\mathrm L^p(\R^d,\R^+),\;\mbox{$V$ radial and}\;\nrm Vp=\alpha\Big\}\,.
\]
Below in Appendix~\ref{appendix:radial}, we provide some numerical evidences that the optimal potentials are radial. We abusively write $V(x)=V(r)$ with $r=|x|$, $x\in\R^2$, use polar coordinates $(x,y)=(r\cos \theta, r \sin \theta)$, and write
\[
\partial_x=\cos\theta\,\partial_r-\frac1r\,\sin\theta\,\partial_\theta\quad \mbox{and}\quad \partial_y=\sin\theta\,\partial_r+\frac1r\,\cos\theta\,\partial_\theta\,.
\]
In these coordinates, the Dirac operator becomes
\[
\Dirac=\begin{pmatrix}
m&-\,\ri\,\partial_x-\partial_y\\
-\,\ri\,\partial_x+\partial_y&-\,m
\end{pmatrix}=
\begin{pmatrix}
m&\re^{-\,\ri\,\theta} \(-\,\ri\,\partial_r-\frac1r\,\partial_\theta\)\\
\re^{ \ri\,\theta} \(-\,\ri\,\partial_r+\frac1r\,\partial_\theta\)&-\,m
\end{pmatrix}.
\]
This suggests to decompose a spinor $\Psi$ in Fourier modes with the convention
\[
\Psi(r, \theta)=\sum_{n\in\Z} \begin{pmatrix}
\varphi_n(r)\,\re^{\ri\,n\,\theta}\\ \ri\,\chi_n(r)\,\re^{ \ri\,(n+1)\,\theta}
\end{pmatrix}\,.
\]
If $\Phi :=(\Dirac-V)\,\Psi$ with corresponding Fourier modes $\big((\widetilde{\varphi_n}, \widetilde{\chi_n})^\top\big)_{n\in\Z}$, then we have
\[
\begin{pmatrix}
\widetilde{\varphi_n}\\ \widetilde{\chi_n}
\end{pmatrix}
=(\Dirac^{(n)}-V) \begin{pmatrix} \varphi_n\\ \chi_n \end{pmatrix}\quad \mbox{with}\quad \Dirac^{(n)}=\begin{pmatrix} m&\partial_r+\frac{n+1}{r}\\-\,\partial_r+\frac{n}{r}&-\,m
\end{pmatrix}\,.
\]
The operator $\Dirac^{(n)}$ is self-adjoint in the Hilbert space $\mathrm L^2\big(\R^+\times (0, 2\pi), r\,\rd r\,\rd \theta\big)$ because $(\partial_r)^*=-\,\partial_r-\frac1r$. Let $\lambda_D^{(n)}(V)$ denote the lowest eigenvalue of $\Dirac^{(n)}-V$ in the gap $(-\,m, m)$, and let
\[
\Lambda_D^{\rad, (n)}(\alpha,p):=\inf\Big\{\lambda_D^{(n)}(V)\,:\,V\in\mathrm L^p(\R^2,\R^+), \; \mbox{$V$ radial and}\;\nrm Vp=\alpha\Big\}
\]
and
\[
\Lambda_D^{\rad}(\alpha, p) :=\inf_{n\in\Z} \Lambda_D^{\rad, (n)}(\alpha, p)\,.
\]
We have the estimates
\be{RadialEstimates:d=2}
\Lambda_D(\alpha, p)\le \Lambda_D^{\rad}(\alpha, p)\le \Lambda_D^{\rad, (0)}(\alpha, p)\,.
\ee
A wavefunction $\Psi(r, \theta)=\big(\varphi(r)\,\re^{\ri\,n\,\theta}, \ri\,\chi(r)\,\re^{ \ri\,(n+1)\,\theta}\big)^\top$ solves the non-linear Dirac equation~\eqref{eq:def:KellerDirac} if and only if
\begin{equation}\label{eq:NLD_2d}
\begin{cases}
\varphi'-\frac nr\,\varphi=-\,\Big(\lambda+m+\(|\chi|^2+|\varphi|^2\)^\frac1{p-1}\Big)\,\chi\,,\\
\chi'+\frac{n+1}r\,\chi=\Big(\lambda-m+\(|\chi|^2+|\varphi|^2\)^\frac1{p-1}\Big)\,\varphi\,.
\end{cases}
\end{equation}
This system with $n=0$ is studied by W.~Borrelli in~\cite{MR3705703}. It is an open question to decide whether $\Lambda_D^{\rad}(\alpha, p)$ is attained by $
\Lambda_D^{\rad, (n)}(\alpha, p)$ with $n=0$ or not, and if equality holds in~\eqref{RadialEstimates:d=2} so that $\Lambda_D(\alpha, p)=\Lambda_D^{\rad, (0)}(\alpha, p)$. See Fig.~\ref{fig:alphac_d2} for some numerical results.
\begin{figure}[ht]
\begin{subfigure}{0.55\textwidth}
\includegraphics[width=1\textwidth]{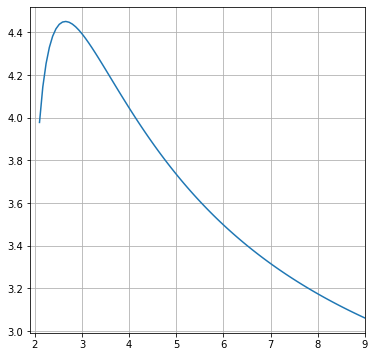}
\end{subfigure}
\begin{subfigure}{0.4\textwidth}
\includegraphics[width=1\textwidth]{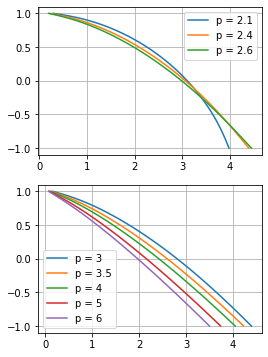}
\end{subfigure}
\caption{Radial case with $d=2$ and $m=1$. (Left) The function $p\mapsto\alpha_\star^{\rad, (n=0)}(p)$ is an upper bound for $\alpha_\star(p)$ and reaches its maximum for $p \approx 2.66$. (Right) The maps $\alpha \mapsto \Lambda_D^{\rad, (n=0)}(\alpha, p)$ for values of $p$ corresponding either to $p<2.66$ (top) or $p>2.66$ (bottom). Numerically the case $n=-1$ gives worse estimates.}
\label{fig:alphac_d2}
\end{figure}

\subsection{The radial case in dimension \texorpdfstring{$d=3$}{d=3}}\label{Sec:Explicit3}

As in the two dimensional case, we restrict the minimization problem~\eqref{eq:min_vap_Dirac} to radially symmetric decreasing potentials. The corresponding Dirac operator decomposes as a direct sum in eigenspaces of the \emph{spin-orbit operator}
\[
K=\beta\,(2\,S \cdot L+1)=\beta\,(J^2-\mathrm L^2+1/4)\,,\quad\operatorname{spec}(K)=\pm 1, \pm 2, \cdots
\]
and the \emph{total angular momentum} in the $z$-direction $J_3$, with $\operatorname{spec}(J_3)=\frac12 \{1,2,3, \cdots \}$. See~\cite[Section~4.6.4]{MR1219537} for details. For any $\kappa\in\operatorname{spec}(K)$, we introduce the operator
\[
\Dirac^{(\kappa)} :=\begin{pmatrix} m-V&\partial_r+\frac{\kappa+1}r\\-\,\partial_r+\frac{\kappa-1} r&-\,m-V\end{pmatrix}
\]
as a self-adjoint operator acting on $\mathrm L^2(\R^+, r^2\,\rd r)$. If$\lambda_D^{(\kappa)}(V)$ denotes the lowest eigenvalue of $\Dirac^{(\kappa)}-V$ in the gap $(-\,m, m)$, let us define
\[
\Lambda_D^{\rad, (\kappa)}(\alpha,p):=\inf\Big\{\lambda_D^{(\kappa)}(V)\,:\,V\in\mathrm L^p(\R^3,\R^+), \; \mbox{$V$ radial and}\;\nrm Vp=\alpha\Big\}\,.
\]
We have $\Lambda_D^{\rad}(\alpha, p)=\inf_{\kappa\in\Z \setminus \{ 0 \}} \Lambda_D^{\rad, (\kappa)}(\alpha, p)$ and
\[
\Lambda_D(\alpha, p)\le \Lambda_D^{\rad}(\alpha, p)\le \Lambda_D^{\rad, (\kappa=1)}(\alpha, p)\,.
\]
It is an open question to decide whether the above inequalities are in fact equalities or not. If $\kappa=1$, we look for an eigenstate of \hbox{$\Dirac-V$} in the Wakano form of~\cite{Wakano_1966}, that is,
\[
\Psi(r,\theta,\phi)=\begin{pmatrix} \varphi(r)\\ 0\\\ri\,\chi(r)\,\cos\theta\\ \ri\,e^{\kern 1pt\ri\kern 1pt\phi}\,\chi(r)\,\sin\theta \end{pmatrix}
\]
so that the nonlinear equation becomes
\be{eq:NLD_3d}
\begin{cases}
\varphi'=-\,\Big((\lambda+m)+\(|\chi|^2+|\varphi|^2\)^\frac1{p-1}\Big)\,\chi\,,\\
\chi'+\frac2r\,\chi=\Big((\lambda-m)+\(|\chi|^2+|\varphi|^2\)^\frac1{p-1}\Big)\,\varphi\,.
\end{cases}
\ee
System~\eqref{eq:NLD_3d} provides us with numerical upper estimates of $\Lambda_D(\alpha,p)$: see Fig.~\ref{fig:alphac_d3}.
\begin{figure}[ht]
\begin{subfigure}{0.55\textwidth}
\includegraphics[width=1\textwidth]{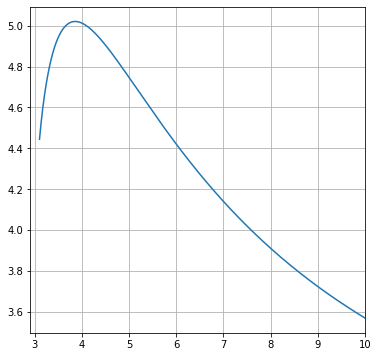}
\end{subfigure}
\begin{subfigure}{0.4\textwidth}
\includegraphics[width=1\textwidth]{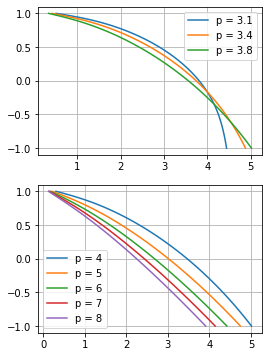}
\end{subfigure}
\caption{Radial case with $d=3$ and $m=1$. (Left) The function $p\mapsto\alpha_\star^{\rad, (\kappa = 1)}(p)$ reaches its maximum at $p\approx3.86$. (Right) The maps $\alpha \mapsto \Lambda_D^{\rad, (\kappa = 1)}(\alpha, p)$ for values of $p$ corresponding either to $p<3.86$ (top) or $p>3.86$ (bottom).}
\label{fig:alphac_d3}
\end{figure}

\subsection{An explicit bound in the radial case in dimensions \texorpdfstring{$d=2$ or $d=3$}{d=2 or d=3}}\label{Sec:Explicit-ODE}

Let us assume that $m=1$ and consider at $\lambda=-1$ (lower end of the gap) the system
\begin{equation} \label{SE}
\varphi'=-\,W\,\chi\,,\quad\chi'+\frac\delta r\,\chi=(W-2)\,\varphi, \,\quad W^{p-1}=|\varphi|^2+|\chi|^2\,.
\end{equation}
According to the previous section, the radial case $d=2$ corresponds to $\delta = 1$ (that is $n = 0$ in~\eqref{eq:NLD_2d}), and the radial case $d = 3$ to $\delta = 2$ (that is $\kappa = 1$ in~\eqref{eq:NLD_3d}). Writing $\chi(r) = f(r)\,\varphi(r) $, the equation becomes
\[
 \varphi' = -\,W\,f\,\varphi\,, \quad f' = W\,(f^2 + 1) - \frac{\delta}{r}\,f - 2\,, \quad W^{p-1} = | \varphi |^2\(1 + |f|^2\)\,.
\]
We now notice that this system admits a solution with $f(r) = r/\mu$ (so that all functions in the middle equality are constant functions). Explicitly, assuming $\delta < p - 1$, with $\mu:=\frac12\,(p-1 -\delta)$, we find a solution of~\eqref{SE} given by
\be{SEsoln}
\varphi_p(r)=\frac{(p\,\mu)^\frac{p-1}2\,\mu}{\(\mu^2+r^2\)^{p/2}}\,, \quad
\chi_p(r)=\frac{(p\,\mu)^\frac{p-1}2\,r}{\(\mu^2+r^2\)^{p/2}}
\quad\mbox{and}\quad W_p(r)=\frac{p\,\mu}{\mu^2+r^2}\,.
\ee
This solution is reminiscent of the solution of~\cite[Corollary~1.4]{MR4239839}. Up to a slight abuse of notations, we can consider $W_p$ as a function of $x\in\R^d$ with $r=|x|$.
\begin{lemma} For all $p \ge d\ge 2$ or $p > 1$ if $d = 1$, and all $\delta < p - 1$, the potential $W_p$ in~\eqref{SEsoln}, seen as a radial function in $\mathrm L^p(\R^d)$, satisfies
\begin{equation} \label{eq:upper_bound_Wp}
\nrm{W_p}p^p=p^p\,\pi^\frac d{2}\,\big(\tfrac2{p-1-\delta}\big)^{p - d}\,\tfrac{\Gamma\(p-\frac d2\)}{\Gamma(p)}
\end{equation}
so that in particular $\lim_{p\to d_+}\nrm{W_p}p= d\,\sqrt{\pi}\,\Big(\tfrac{\Gamma\(\frac d2\)}{\Gamma(d)}\Big)^{1/d}$.
\end{lemma}
\noindent Applied either with $d=2$ and $\delta=1$, or $d=3$ and $\delta=2$, the expression~\eqref{eq:upper_bound_Wp} gives an upper bound for $\alpha_*(p)$ in dimension $d = 2$ and $d = 3$. We find that
\begin{align*}
\alpha_\star(p)^p & \le\tfrac{p^p}{p-1} \,\big(\tfrac2{p-2}\big)^{p-2}\,\pi \quad\mbox{if}\quad d=2\,,\\
\alpha_\star(p)^p & \le p^p \,\big(\tfrac2{p-3}\big)^{p-3}\,\tfrac{\Gamma\(p-\frac 32\)}{\Gamma(p)}\,\pi^\frac 3{2}\quad\mbox{if}\quad d=3\,.
\end{align*}
In particular, 
\begin{align*}
&\alpha_\star^{\rad,(n=0)}(2)\le2\,\sqrt\pi\approx3.54491\quad\mbox{if}\quad d=2\,,\\
&\alpha_\star^{\rad,(\kappa=1)}(3)\le3\(\tfrac\pi2\)^{2/3}\approx4.05385\quad\mbox{if}\quad d=3\,.
\end{align*}
The upper bound given by this expression for $d = 1$ (with $\delta = 0$) coincides with the expression found in Theorem~\ref{th:alphac_d=1}, and we conjecture that we actually have equality in $d = 2$ and $d = 3$ as well. Numerically, the curve $p\mapsto\nrm{W_p}p$ coincides with the numerical solution $p\mapsto\alpha_\star^{\rad,(n=0)}(p)$ if $d=2$ and $p\mapsto\alpha_\star^{\rad,(\kappa=1)}(p)$ if $d=3$ of Figs.~\ref{fig:alphac_d2} and~\ref{fig:alphac_d3}. It is however an open question to decide whether $\varphi_p$, $\chi_p$ and $W_p$ is the unique solution of~\eqref{SE} and if it is optimal among radial optimal functions, and also among non-radial optimal functions (see Appendix~\ref{appendix:radial}).

\begin{center}\large{\bf Appendices}\end{center}
\appendix

\section{Open questions}\label{appendix:open}

In this article, we study the \emph{ground state} defined the {\em lowest eigenvalue in the gap} $\lambda_D(V)$ of a general Dirac operator $\Dirac - V$ with $V \in\mathrm L^p(\R^d, \R^+)$ using Birman-Schwinger techniques, and prove that this quantity always makes sense if the $\mathrm L^p(\R^d)$ norm of $V$ is small enough. To our knowledge, there are several open questions concerning this lowest eigenvalue, which we recall here.
\begin{itemize}
 \item Is the map $V \mapsto \lambda_D(V)$ concave?
 \item Is $\lambda_D(V)$ always a simple eigenvalue, or equivalently, is $\mu_1(K_V)$ always simple?
\end{itemize}
Assuming that the answer of the last question is positive, we denote by $\Psi$ the corresponding eigenfunction for the Dirac operator. We decompose it as $\Psi = \Psi_+ + \Psi_-$ with $\beta\,\Psi_+ = \Psi_+$ (upper component) and $\beta\,\Psi_- = -\, \Psi_-$ (lower component).
\begin{itemize}
 \item If $V$ is radial (decreasing), is $\Psi_+$ also radial (decreasing)?
\end{itemize}
Concerning the variational problem associated with~\eqref{eq:min_vap_Dirac}, we recall two questions that were already raised earlier:
\begin{itemize}
 \item Is the optimal potential $V$ radial (decreasing) if $d\ge2$?
 \item If so, is the corresponding \emph{ground state} $\Psi$ the solution with lowest angular momentum and smallest number of oscillations, as it is suggested in Sections~\ref{Sec:Explicit2} and~\ref{Sec:Explicit3}?
\end{itemize}

\section{Is the optimal potential radial? A numerical answer}
\label{appendix:radial}

In dimension $d=2$, we investigate numerically whether the optimal potential $V$ for~\eqref{eq:min_vap_Dirac} is radial, or equivalently whether the optimal potential $W$ for~\eqref{eq:def:N} is radial. In order to do so, we run the following self-consistent algorithm\footnote{The code is available upon request to the authors.}. Recall that $K_W :=\sqrt W\,R_0(\lambda)\,\sqrt W$ where $R_0$ denotes the resolvent of the free Dirac operator. For $p > d=2$ and $\lambda\in [-\,m, m)$, we choose an initial potential $W_0$ at random, and set
\[
\begin{cases}
\phi_k&:=\kern 6pt \mbox{normalized eigenvector corresponding to $\mu_1\(K_{W_k}\)$}\,,\\
W_{k+1}&:=\kern 6pt | \phi_k |^{2/p}\,.
\end{cases}
\]
In practice, the potential $W_{k+1}$ is also translated so that its maximum is at the origin. We can check that the quantity $\mu_1(K_{W_k})$ is increasing, and that the sequence $(W_k)_{k\in\N}$ converges to some limit potential $W_*$ in $\mathrm L^p(\R^2)$. A typical run of the algorithm is displayed in Fig~\ref{fig:movie}.
\begin{figure}[ht]
\includegraphics[width=0.9\textwidth]{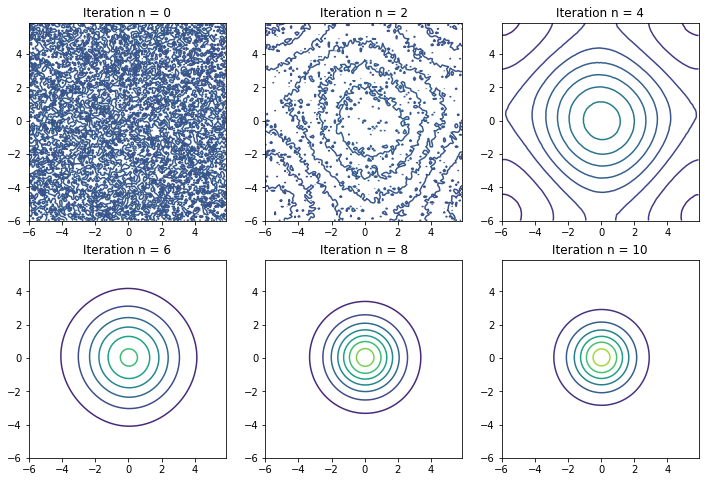}
\caption{Contour lines of the potential $W_k$ during the iterations, for $p=3$ and $\lambda=1/2$, for some~$W_0$ chosen at random. The quantities $W_k$ and $\phi_k$ are computed on a square $[-a, a]^2$ with $a=6$, \hbox{$L=100$} discretization points per direction and periodic boundary conditions. The Dirac operator and its inverse are computed in Fourier space and the $\mathrm L^p(\R^2)$ integrals in direct space.}
\label{fig:movie}
\end{figure}
In order to check whether $W_*$ is radial or not, we compute the $\mathrm L^p(\R^2)$ norm of its angular derivative. For $\lambda\in [-0.9, 0.9]$, $m=1$ and $p\in (2, 8)$, this norm is always much smaller than~$1$ and usually of the order of $10^{-2}$ or $10^{-3}$, after less than 100 iterations, depending on the parameters we chose. These numerical results suggest that the optimal potentials might be radial, up to translations.

\section{A nonlinear interpolation inequality for the Dirac operator}
\label{appendix:GNS}

\subsection{Non-relativistic limit and Keller-Lieb-Thirring inequalities}\label{appendix:GNS-1}

In order to consider the \emph{non-relativistic limit} $c\to+\infty$, it is interesting to reintroduce the parameters $\hbar$, $m$ and $c$. The eigenvalue problem
\[
\(\Dirac^{\hbar,c}-W\)\psi=\mu\,\psi\quad\mbox{where}\quad\Dirac^{\hbar,c}:=-\,\ri\,\hbar\,c\,\boldsymbol\alpha\cdot\nabla+m\,c^2\beta
\]
is reduced to the eigenvalue problem corresponding to $\hbar=c=m=1$ by the change of variables
\[
\psi(x)=\Psi\(\frac{m\,c}\hbar\,x\)\,,\quad W(x)=m\,c^2\,V\(\frac{m\,c}\hbar\,x\)\,,\quad\mu=m\,c^2\,\lambda\,.
\]
As a consequence, the \emph{ground state} $\lambda_D^{\hbar,c}(W)$ of $\Dirac^{\hbar,c}-W$ defined as its lowest eigenvalue in the gap $\(-m\,c^2,m\,c^2\)$ and estimated by $\lambda_D(V)\ge\Lambda_D^{(m=1)}\(\nrm Vp,p\)$ according to the \emph{Keller-Lieb-Thirring inequality for the Dirac operator}~\eqref{KLT-Dirac} becomes
\be{KLThbarc}
\lambda_D^{\hbar,c}(W)\ge m\,c^2\,\Lambda_D^{(m=1)}\(\hbar^{-\frac dp}\,m^{\frac dp-1}\,c^{\frac dp-2}\,\nrm Wp,p\)
\ee
using the above change of variables. Here $\Lambda_D^{(m=1)}$ stands for $\Lambda_D$ when we assume $m=1$ in notations of Theorem~\ref{Thm:Main1}.
\begin{proposition}\label{Prop:NR} Let either $d\ge1$ and $p>1$ if $d=1$, or $p\ge d$ if $d\ge2$. With $\Lambda_D^{(m=1)}(\alpha,p)$ defined by~\eqref{eq:min_vap_Dirac}, $\eta=2\,p/(2\,p-d)$ and $\mathsf K_p$ as in~\eqref{KLT}, we have
\[
1-\Lambda_D^{(m=1)}(\alpha,p)=2^\frac d{2\,p-d}\,\mathsf K_p\,\alpha^\eta\,\big(1+o(1)\big)\quad\mbox{as}\quad\alpha\to0_+\,.
\]
\end{proposition}
\noindent If $d=1$, we obtain that
\[
\mathsf K_p=\big(p^p\,(p-1)^{-(p-1)}\,B(\tfrac12,p)\big)^{-\frac2{2\,p-1}}
\]
by expanding the expression of $\alpha_D(\lambda,p)$ given in Theorem~\ref{th:alphac_d=1} as $\lambda\to1_-$. This is consistent with $\mathsf K_p=\mathsf C_q^{-\eta}$ and the expression of the explicit, optimal value of the constant $\mathsf C_q$ in~\eqref{ScaledGN} if the dimension is $d=1$: we refer to~\cite{Dolbeault06082014} and references therein for details.

\begin{proof} Let us consider the general case $d\ge1$. The non-relativistic limit of the ground state $\lambda_D^{\hbar,c}(W)$ of the Dirac operator \hbox{$\Dirac^{\hbar,c}-W$} is, up to the mass energy $m\,c^2$, given by the ground state of the Schr\"odinger operator
\[
-\,\tfrac{\hbar^2}{2\,m}\,\Delta-W
\]
by standard results: see for instance~\cite[Section~2.4]{Esteban_2008}. Hence
\[
\lim_{c\to+\infty}\(m\,c^2-\lambda_D^{\hbar,c}(W)\)=\lambda_S^-(W_\mu)\quad\mbox{where}\quad W_\mu(x):=W(\mu\,x)\quad\mbox{and}\quad\mu=\frac\hbar{\sqrt{2\,m}}\,.
\]
Here $-\,\lambda_S^-(W_\mu)$ denotes, if it exists, the negative ground state of
the Schr\"odinger operator \hbox{$-\,\Delta-W_\mu$}. The factor $\mu=\hbar/\sqrt{2\,m}$ arises from a scaling argument. By definition~\eqref{eq:min_vap_Dirac}, we obtain
\[
\lim_{c\to+\infty} m\,c^2\(1-\Lambda_D^{(m=1)}\(\hbar^{-\frac dp}\,m^{\frac dp-1}\,c^{\frac dp-2}\,\nrm Wp,p\)\)\le\mathsf K_p\,\nrm{W_\mu}p^\eta=\mathsf K_p\,\mu^{-\frac{d\,\eta}p}\,\nrm Wp^\eta
\]
but there is in fact equality if we use as test function an optimal function $W$ for~\eqref{KLT}. Taking $\alpha=\hbar^{-\frac dp}\,m^{\frac dp-1}\,c^{\frac dp-2}\,\nrm Wp$ in the limit as $c\to+\infty$ concludes the proof.
\end{proof}

Proposition~\ref{Prop:NR} is in fact equivalent to
\begin{equation} \label{NR}
\lim_{c\to+\infty}\(m\,c^2-\lambda_D^{\hbar,c}(W)\)\le\mathsf K_p\(\tfrac{2\,m}{\hbar^2}\)^\frac d{2\,p-d}\,\nrm Wp^\eta
\end{equation}
written with the physical constants. In other words, we recover a standard \emph{Keller-Lieb-Thirring inequality for the Schr\"odinger operator}~\eqref{KLT} in the non-relativistic limit. In dimension $d=1$, a tedious but elementary computation directly shows that the constant obtained by taking the non-relativistic limit in the Keller-Lieb-Thirring inequality for the Dirac operator written with optimal constant is the optimal constant in the Keller-Lieb-Thirring inequality for the Schr\"odinger operator, as it can be deduced for instance from~\cite{MR121101,Dolbeault06082014}.

The definition~\eqref{eq:def:alpha0} can be generalized to the case $(\hbar,c)\neq(1,1)$ using the monotonicity of $\alpha\mapsto\Lambda_D^{(m=1)}(\alpha,p)$ stated in Theorem~\ref{Thm:Main1} and~\eqref{KLThbarc}. If $\alpha_D^{(m=1)}$ denotes the inverse of $\alpha\mapsto\Lambda_D^{(m=1)}(\alpha,p)$, the condition
\be{DefalphaDhbarc}
\|W\|_p\le\alpha_D^{\hbar,c}(\lambda,p):=\hbar^\frac dp\,m^{1-\frac dp}\,c^{2-\frac dp}\,\alpha_D^{(m=1)}\!\left(\tfrac\lambda{m\,c^2},p\right)
\ee
guarantees that $\lambda_D^{\hbar,c}(W)\ge\lambda$. Notice that $p\ge d$ implies that
\[\label{alphac}
\lim_{c\to\infty}\|W\|_p\le\alpha_D^{\hbar,c}(\lambda,p)=\infty\,.
\]

\subsection{An interpolation inequality for the Dirac operator}\label{appendix:GNS-2}

Using a min-max principle as in~\cite{MR1761368}, it is possible to write an optimal interpolation inequality of Gagliardo-Niren\-berg-Sobolev type which plays for the free Dirac operator the same role as~\eqref{ScaledGN}. The inequality is somewhat involved, but Inequality~\eqref{ScaledGN} is reco\-vered in the non-relativistic limit as $c\to+\infty$. For sake of simplicity, we consider only the case~$d=1$.

Let us start by a short and formal summary of the \emph{min-max principle} applied to the determination of the ground state of the Dirac operator. If $(\varphi,\chi)^\top$ is an eigenspinor of the operator $\Dirac^{\hbar, c}-V$ with eigenvalue $\lambda \in(-\,m \, c^2,\,m\,c^2)$, then, as in~\eqref{eq:system-1-d} we have
\[
\begin{cases}
\hbar\,c\,\varphi' & = -\,(\lambda + m \, c^2 + V)\,\chi\,, \\
\hbar\,c\,\chi' & = (\lambda - m \, c^2 + V)\,\varphi\,.
\end{cases}
\]
The first line gives
\[
\chi = -\,\hbar\,c\,\frac{\varphi'}{ \lambda + m \, c^2 + V}
\]
so that the problem amounts to solving
\[
-\,(\hbar\,c)^2 \( \frac{\varphi'}{\lambda + m \, c^2 + V}\)'+ \( m \, c^2 - \lambda - V \) \varphi = 0\,.
\]
Multiplying by $\varphi$ and integrating suggests to introduce the functional
\[
\mathcal E[\mu,V,\phi]:= (\hbar\,c)^2 \ir{\frac{|\phi'|^2}{\mu + m \, c^2 +V}}+\ir{\(m \, c^2-\mu-V\)|\phi|^2}\,.
\]
Clearly, we have $\cE[\lambda, V, \varphi] = 0$. In addition, for all fixed $V$ and $\phi$, the map $\mu \mapsto \cE[\mu, V, \phi]$ is decreasing. It is proved in~\cite[Lemma 2.4]{SchSolTok20}, that for all $-\,m < \mu < \lambda_D^{\hbar, c}(V)$, the quadratic map $\phi \mapsto \cE[\mu, V, \phi]$ is positive definite, and that, for $\mu = \lambda_D(V)$, we have $\cE[\mu, V, \phi] = 0$ if and only if $\phi=\varphi$, up to a multiplicative constant. In particular, we have
\[
 \forall\,\phi \in C^\infty_0(\R)\,, \quad \forall\,V \in\mathrm L^p(\R)\,, \quad \| V \|_p \le \alpha_D^{\hbar, c}(\lambda, p) \implies \cE[\lambda,V,\phi] \ge 0\,.
\]
Minimizing $\cE[\lambda,W,\phi]$ in $W$ such that $\|W\|_p=\alpha\le \alpha_D^{\hbar, c}(\lambda, p)$ shows that the optimal $W$ solves the Euler-Lagrange equation of the implicit form
\be{eq:implicit_V}
 \nu\,W^{p-1} = | \phi |^2 + \dfrac{(\hbar\,c)^2\,| \phi' |^2}{(\lambda + m \, c^2 + W)^2}\,,
\ee
where $\nu \ge 0$ is now the Lagrange multiplier for the constraint $\|W\|_p=\alpha$. Note that for all fixed $a$, $b$, $c \ge 0$, the equation
\[
 \nu\,X^{p-1} = a + \frac{b}{(c + X)^2}
\]
has a unique solution in $X_\nu \ge 0$, as the left-hand side is an increasing function of $X$, while the right-hand side is decreasing, and that $\nu \mapsto X_\nu$ is increasing. So for fixed $\nu \ge 0$, there is a unique $W = V_\nu[\phi]$ satisfying~\eqref{eq:implicit_V} and the map $\nu \mapsto V_\nu$ is pointwise decreasing, hence so is the map $\nu \mapsto \| V_\nu \|_p$. With $\alpha_D^{\hbar, c}(\lambda, p)$ given by~\eqref{DefalphaDhbarc}, we define
\[
 \nu_*(\lambda, p, \phi) := \inf \left\{ \nu > 0\,:\,\| V_\nu[\phi] \|_p \le \alpha_D^{\hbar, c}(\lambda, p) \right\}\,.
\]
Summarizing, we proved that for all $\phi \in C^\infty_0(\R)$ and all $\nu \ge \nu_*(\lambda, p, \phi)$, 
\be{eq:GNineq_Dirac}
 (\hbar\,c)^2 \ir{\frac{|\phi'|^2}{\lambda + m \, c^2 +V_\nu[\phi]}}+\ir{\Big(m \, c^2-\lambda-V_\nu[\phi]\Big)\,|\phi|^2} \ge 0\,,
\ee
which can be interpreted as a \emph{Gagliardo--Nirenberg type inequality} for $\phi$ alone. Such an inequality is known for a fixed, given potential $V$ from~\cite{MR1761368,DDEV,MR2091354} and it is then of Hardy-type, as for instance the new Hardy inequality in~\cite{esteban2021diraccoulombII}, but the novelty in this paper is that we take $V=V_\nu[\phi]$ thus making it a non-linear interpolation inequality.
While the form~\eqref{eq:GNineq_Dirac} is non-explicit, it allows to recover the usual Gagliardo--Niren\-berg inequality in the non-relativistic limit as $c \to \infty$. By writing $\lambda = m\,c^2 + E$ for some $E < 0$,~\eqref{eq:GNineq_Dirac} becomes
\[
 (\hbar\,c)^2 \ir{\frac{|\phi'|^2}{2\,m \,c^2 + E +V_\nu[\phi]}}- \ir{\big(E + V_\nu[\phi]\big)\,|\phi|^2} \ge 0\,.
\]
Let us choose $\nu=\|\phi\|^2_{2\,p/(p-1)}$. 
As $c \to \infty$, we get from~\eqref{eq:implicit_V} that $V_\nu[\phi]$ converges to $\|\phi\|^{-2/(p-1)}_{2\,p/(p-1)}\,|\phi|^{2/(p-1)}$. Together with~\eqref{NR}, we get that $\nu \ge \nu_*(\lambda, p, \phi)$ in the limit $c \to \infty$ whenever $|E|\ge\mathsf K_p\,(2\,m/\hbar^2)^{d/(2\,p-d)}$. We obtain
\begin{equation*}
 \dfrac{\hbar^2}{2\,m} \ir{|\phi'|^2} - \| \phi \|_{\frac{2\,p}{p-1}}^2 \ge E \ir{| \phi |^2}\,.
\end{equation*}
This inequality is the Gagliardo--Nirenberg inequality~\eqref{ScaledGN} written in non-scale invariant form, for an appropriate choice of the parameter $\lambda$ in~\eqref{ScaledGN}.

\section{The case \texorpdfstring{$p=d=1$}{p=d=1}}\label{App:D}

This appendix deals with the limit case $p=1$ of Theorem~\ref{th:alphac_d=1} devoted to the one-dimension\-al Keller estimates. We give a computation of $\alpha_D(\lambda,1)$ which is not based on the limit as $p\to1_+$ of the nonlinear estimates and prove that any sequence of optimizing potentials concentrates into a Dirac $\delta$ distribution. 
\begin{proposition}\label{Prop:p=d=1}
If $d=1$, then $\alpha_D(\lambda,1)=\arccos(\lambda/m)$. More specifically, for all $\alpha \in (0, \pi)$, all $V \in \mathrm L^1(\R, \R^+)$ with $\nrm V1 = \alpha$, if $\lambda \in (-m,m)$ is an eigenvalue of $\Dirac - V$, then we have the strict inequality
$$
m\,\cos\alpha < \lambda\,.
$$
In addition, any sequence of nonnegative potentials $(V_n)_{n\in\N}$ with $\nrm{V_n}1 = \alpha$ and eigenvalues $\lambda_n$ approaching $m\,\cos\alpha$, converges as $n\to+\infty$ to a Dirac $\delta$ distribution.
\end{proposition}
According to~\cite{MR1009526}, ``the method of directly solving the Dirac equation with a $\delta$-function potential and the method of obtaining the solution by first solving the Dirac equation with a short-range potential and afterward taking the $\delta$-function limit, lead to different results'' [concerning the spectrum]. This issue is known as Klein's paradox. Although the Keller-Lieb-Thirring~\eqref{cos} makes sense for any nonnegative potential $V\in\mathrm L^1(\R^d)$, it is a natural question to investigate by direct methods whether the bound is achieved in the larger set of bounded nonnegative measures and consider sequences of optimizing potentials.

\begin{proof} We start with a calculation for a bounded and compactly supported potential $V$. In this case, the eigenvalue equation rewrites as
\[
\Psi' = \big(\ri\,\sigma_2\,(V + \lambda)+ m\,\sigma_1\big)\,\Psi\,.
\]
We decompose $\Psi$ on the (not-orthonormal) basis given by the eigenvectors $e_\pm$ of the matrix $\ri\,\lambda\,\sigma_2 - m\,\sigma_1$ defined by
\begin{equation*}
e_\pm := \begin{pmatrix}\sqrt{m^2 - \lambda^2}\,\\ 
\pm\,(m-\lambda) 
\end{pmatrix} \quad \text{such that}\quad \begin{pmatrix}
0 & m+\lambda  \\ m-\lambda  &0
\end{pmatrix} e_\pm = \pm\,\sqrt{m^2 - \lambda^2} \, e_\pm\,.
\end{equation*}
Decomposing $\Psi (x) = \aalpha (x)\,e_+ + \bbeta (x)\,e_- $ and using the identities 
\begin{align*}
    &\langle \ri\,\sigma_2\,e_\pm, e_\pm \rangle = 0\,,
    \quad& \langle \ri\,\sigma_2\,e_\pm, e_\mp \rangle = \pm\,2\,(m-\lambda)\,\sqrt{m^2 -\lambda^2}\,,\\
    &\langle \ri\,\sigma_2\,e_\pm, \ri\,\sigma_2\,e_\pm \rangle =  2\,m\,(m-\lambda)\,, 
    &\quad \langle \ri\,\sigma_2\,e_\pm,\ri\,\sigma_2\,e_\mp \rangle =  2\,\lambda\,(m-\lambda)\,,
\end{align*}
gives, with $\sV := V/\sqrt{m^2-\lambda^2}$,
\begin{align*}
\aalpha ' &= \left(\sqrt{m^2 -\lambda^2}- \lambda\,\sV \right) \aalpha - m\,\sV\,\bbeta\,, \\
\bbeta' &=- \left( \sqrt{m^2 -\lambda^2} - \lambda\,\sV \right) \bbeta + m\,\sV\,\aalpha\,.
\end{align*}
Since $V$ (and $\sV$) are compactly supported, a square-integrable solution must have $\bbeta=0$ in a neighborhood of $-\infty$ and $\aalpha = 0$ in a neighborhood of $+\infty$.  Without loss of generality, we take a solution with $\aalpha(x) > 0$ for $x$ near $-\infty$. 
Since 
\[
(\aalpha\,\bbeta) ' = m\,\sV\,(\aalpha^2 - \bbeta^2) 
\]
is nonnegative if $|\aalpha| > |\bbeta|$,
such a solution enters the first $(\aalpha,\bbeta)$ quadrant and stays in the first quadrant until the first value of $x$ such that $\aalpha(x) = 0$. We denote this value by $x_1$, with $x_1= +\infty$ if $\aalpha$ does not change sign.
In the interval $(-\infty, x_1)$,  the ratio $\st:= \bbeta/\aalpha$ is well-defined and satisfies
\begin{multline*}
\st' = \frac{1}{\aalpha^2} \left( -\,2\,\aalpha\,\bbeta \left(\sqrt{m^2 -\lambda^2}- \lambda\,\sV \right) + m\,\sV \(\aalpha^2 +\bbeta^2\) \right) \\
= -\,2\,\sqrt{m^2 -\lambda^2}\,\st + \sV \left( m + m\,\st^2 + 2\,\lambda\,\st \right) \,.
\end{multline*}
We finally define the angle 
$$\theta_\lambda(\st):= \arctan \left(\frac{ m\,\st + \lambda}{\sqrt{m^2-\lambda^2}}\right)\,,
$$
such that $\lim_{x \to x_1} \theta_\lambda(\st(x)) = \pi/2 $ and
\be{ODEtheta}
   \frac{1}{\sqrt{m^2 -\lambda^2}}\, (\theta_\lambda \circ \st)'
 =\frac{\st'}{m + 2\,\lambda\,\st + m\,\st^2} 
 = \sV - 2\,\frac{\sqrt{m^2 - \lambda^2}\, \st}{m + 2\,\lambda\,\st + m\,\st^2}\,.
\ee
Integrating  for $x \in (-\infty, x_1)$, we obtain
\begin{align} \label{eq:integral_eq_for_theta}
\frac{\pi/2- \theta_\lambda(0)}{\sqrt{m^2 -\lambda^2} }  
=  \int_{-\infty}^{x_1} \sV(s)\,\rd s -2  \int_{-\infty}^{x_1} \frac{\sqrt{m^2 - \lambda^2}\, \st}{m + 2\,\lambda\,\st + m\,\st^2}\,\rd s 
< \frac{\alpha}{\sqrt{m^2- \lambda^2}}\,.
\end{align}
Since $\theta_\lambda(0) = \arcsin(\lambda/m)$, we obtain
$$
\arccos(\lambda/m) < \alpha \quad \text{or}\quad \lambda > m\,\cos\alpha\,.
$$

In order to approximate unbounded potentials, we need an estimate on the negative term in~\eqref{eq:integral_eq_for_theta}. Take any number $c >1$. Since $\st$ is continuous, there is an interval $I_V(c) \subset (-\infty, x_1]$ such that $\st(x)\in (1/c, c)$ for all $x \in I_V(c)$. 
We have the bound (note that the integrand is symmetric under $\st \mapsto 1/\st$)
$$
\int_{I_V(c)} \frac{ 2\, (m^2 - \lambda^2)\,\st(s)}{m + m\,\st(s)^2 + 2\,\lambda\,\st(s)}\,\rd s \ge\frac{ 2\,(m^2 - \lambda^2)\,c}{m + m\,c^2 + 2\,\lambda\,c}\,|I_V(c)|
$$
and therefore
\begin{equation} \label{eq:inequality_with_error_term}
    \arccos(\lambda/m) \le \alpha - \frac{ 2\,(m^2 - \lambda^2)\,c}{m + m\,c^2 + 2\,\lambda\,c}\,|I_V(c)|\,.
\end{equation}
To prove that $|I_V(c)|$ cannot be arbitrarily small, we integrate~\eqref{ODEtheta} on $I_V(c)$, which gives
\begin{align} \label{eq:lower-bound-for-Q_n}
   \theta_\lambda(c) - \theta_\lambda(1/c) \le \int_{I_V(c)} V(s) \, \rd s \le Q_V\big(|I_V(c)|\big)
\end{align}
where we have defined
$$
Q_V(r) := \sup_{x \in \R} \int_{x-r/2}^{x+r/2} V(s) \, \rd s\,.
$$
Now, assume that $(V_n)_{n\in\N}$ is a sequence of potentials with $\nrm{V_n}1 = \alpha$ and eigenvalues~$\lambda_n$ converging to $\lambda:= m\,\cos\alpha$. Without loss of generality, we may assume that each $V_n$ is bounded and compactly supported. By~\eqref{eq:inequality_with_error_term}, in order to approach the equality case, we need that $|I_{V_n}(c)|$ tends to zero for each $c>1$. We now use~\eqref{eq:lower-bound-for-Q_n} to show that this implies the convergence (after suitable translations) to a Dirac $\delta$ distribution.

Fix $\epsilon>0$ and $r>0$. Fix $c>1$ and $n_0$ such that for all $n\ge n_0$, we have
$$
\theta_{\lambda_n}(c) - \theta_{\lambda_n}(1/c) \ge \theta_\lambda(+\infty) - \theta_\lambda(0)-\epsilon = \alpha -\epsilon\,.
$$
Upon increasing $n_0$, we can assume $|I_{V_n}(c)|\le r$ for all $n \ge n_0$. From~\eqref{eq:lower-bound-for-Q_n}, this gives
$$
Q_{V_n}(r) \ge Q_{V_n}(|I_{V_n}(c)|) \ge \alpha - \epsilon\,.
$$
Since $r$ and $\epsilon$ are arbirary, we have shown that $Q_{V_n}$ converges pointwise to $\alpha$.
In the language of concentration-compactness, this excludes \emph{vanishing} and \emph{dichotomy} and implies that, after a sequence of translations, $V_n$ converges to a measure of total mass $\alpha$ supported at the origin, hence, to a Dirac $\delta$ distribution.
\end{proof}

\begin{ack}
The authors thank the referees for their suggestions which have all been taken into account and led to a significant improvement of the paper.
\end{ack}
\newpage
\begin{funding} This work was partially supported by the Project EFI (ANR-17-CE40-0030) of the French National Research Agency (ANR). HVDB received funding from the Center for Mathematical Modeling (Universidad de Chile \& CNRS IRL 2807) through ANID/Basal projects \#FB210005 and \#ACE210010, as well as
ANID/Fondecyt project \#11220194, and MathAmSud project EEQUADDII 20-MATH-04. This work was partially developed when FP was employed at CNRS \& CEREMADE-Universit\'e Paris-Dauphine and supported by the project ANR-17-CE29-0004 molQED of the ANR and by the European Research Council (ERC) under the European Union's Horizon 2020 research and innovation programme (grant agreement MDFT No.~725528). He is also supported by the project PID2021-123034NB-I00 funded by MCIN/AEI/10.13039/501100011033/ FEDER, UE.

\noindent{\scriptsize \copyright\,\the\year~by the authors. Reproduction of this article by any means permitted for noncommercial purposes. \hbox{\href{https://creativecommons.org/licenses/by/4.0/legalcode}{CC-BY 4.0}}}
\end{funding}


\end{document}